\documentclass[ejs]{imsart}

\RequirePackage[OT1]{fontenc}
\RequirePackage{amsmath,amsfonts,amssymb,amsthm,verbatim,multirow,url,subfig,footnote,graphicx,array,xr,booktabs,placeins}
\RequirePackage[numbers]{natbib}
\RequirePackage[colorlinks,citecolor=blue,urlcolor=blue]{hyperref}

\pubyear{2016}
\volume{0}
\issue{0}
\firstpage{1}
\lastpage{43}

\startlocaldefs
\numberwithin{equation}{section}
\theoremstyle{plain}

\newtheorem{lema}{Lemma}[section]

\newtheorem{teo}{Theorem}[section]
\newtheorem{prop}{Proposition}[section]
\theoremstyle{definition}
\newtheorem{ap}{Assumption}[section]
\newtheorem{defin}{Definition}[section]
\theoremstyle{remark}
\newtheorem*{rk}{Remark}
\newcommand{\nn}{\mathbf}
\newcommand{\nns}{\boldsymbol}
\newcommand{\Hcal}{\mathcal H}
\endlocaldefs

\begin{document}

\begin{frontmatter}
\title{Parameter
  Estimation of Gaussian Stationary Processes using the Generalized Method of Moments.}
\runtitle{Parameter Estimation of Gaussian Stationary Processes}

\begin{aug}
\author{\fnms{Luis A.} \snm{Barboza}
\ead[label=e1]{luisalberto.barboza@ucr.ac.cr}}

\address{Centro de Investigaci\'on en Matem\'atica Pura y Aplicada (CIMPA) \\
Universidad de Costa Rica\\
San Jos\'e, Costa Rica.\\
\printead{e1}}

\author{\fnms{Frederi G.} \snm{Viens}
\ead[label=e2]{
viens@stt.msu.edu}}

\address{Department of Statistics and Probability, Michigan State University\\
East Lansing, MI, USA.\\
\printead{e2}
}

\runauthor{Barboza, L. and Viens, F.}


\end{aug}

\begin{abstract}
  We consider the class of all stationary Gaussian process with explicit parametric spectral
  density. Under some conditions on the autocovariance function, we defined a
  GMM estimator that satisfies consistency and asymptotic normality, using the Breuer-Major 
  theorem and previous results on ergodicity.  This result is
  applied to the joint estimation of the three parameters of a
  stationary Ornstein-Uhlenbeck (fOU) process driven by a fractional Brownian
  motion. The asymptotic normality of its GMM estimator applies for any $H$ in
  $(0, 1)$ and under some restrictions on the remaining parameters.
  A numerical study is performed in the fOU case, to illustrate the estimator's
  practical performance when the number of datapoints is moderate.
\end{abstract}

\begin{keyword}[class=MSC]
\kwd[Primary ]{62M09}
\kwd{62F10}
\kwd[; secondary ]{62F12}
\end{keyword}

\begin{keyword}
\kwd{Fractional Brownian Motion}
\kwd{Ornstein Uhlenbeck process}
\kwd{Method of Moments}
\end{keyword}
\tableofcontents
\end{frontmatter}

\section{Introduction}\label{sec:intro}
In this article we are interested in the parameter estimation of a stationary Gaussian process with explicit spectral density, which applies in particular to the stationary solution of the Ornstein-Uhlenbeck equation driven by the fractional Brownian motion, also known as the fOU process. The main assumption on the data is that a single trajectory of the process is observed at predetermined discrete times, with a fixed length of observation time, and an increasing horizon. Thus this study, which applies to continuous-time Gaussian stochastic processes, also falls within the realm of Gaussian time series with complex parametric auto-covariances.   

It is well known that the Maximum Likelihood estimation (MLE) is the preferred method of estimation because of its asymptotic optimality among densities with regular properties and independent samples (see for example, \cite{Schervish1997}). If the samples are generated from a single path of a stationary process, they can be far from independent, and the asymptotic properties of the MLE can be deduced, but this would depend on their distributional assumptions. For example, one of the most studied classes of stochastic processes is the Gaussian one, mainly because of its wide use in modeling events which have some degree of temporal dependence, particularly when joint-distributional information need only rely on the first two moments, since no further information is needed in the Gaussian case. In a particular instance, the MLE estimator of a stationary ARMA process is strongly consistent and asymptotically efficient (see section 10.8 in \cite{libroamarillo}). As an interesting generalization of the above fact, \cite{Pham-Dinh-Tuan1977} studies the parameter estimation of a Gaussian process with rational spectral density, using a modification of the likelihood function.  

Another interesting example in which the MLE uses spectral information is the estimation of self-similar processes. The best known example of such processes in a Gaussian setting is the fractional Brownian motion (fBm). Its applications have been widely studied in many areas including hydrology, finance, climatology, and others. Jan Beran explains in more detail its definitions and properties in \cite{Beranbook} as well as another example of interest in applications: the fARIMA process. The estimation of the fractal dimension $H$ of the fBm process, also known as Hurst parameter, which turns out to be the same parameter determining the slow power speed of decay of the auto-correlation function for fBm's stationary increments, namely $n^{2H-2}$, has been researched from different points of view: exact and approximate Maximum Likelihood using spectral information (\cite{Fox1986}, \cite{dahlhaus1989} and \cite{Lieberman2011}), discrete variations of the sample paths (\cite{istas}, \cite{Coeurjolly2001}, \cite{TV2}, \cite{TV1} and \cite{CTV2}), and wavelets (\cite{Bardet2000}, \cite{Flandrin1992}, \cite{Moulines2007}). 

A natural generalization of the above example is to consider processes whose variogram is (see \cite{istas}):
\begin{align*}
  v(t)=C|t|^s+o(|t|^s).
\end{align*}
If $s=2H$, a sufficient condition for the above to be true is that the spectral density can be written as a power law with degree $-1-2H$ plus a remainder (see equation (5) in \cite{bierme2008estimation}). It is possible to estimate the parameter $H$ of such processes using the quadratic variation of a filtered process, which will be a fundamental idea in our suggested method, though the interpretation of $H$ as a memory parameter for discrete observation, as in our study (where what matters is the small-$t$ behavior of the variogram), is not the same as its interpretation for use with in-fill asymptotics, where how $H$ determines self-similarity or path regularity is what relates it to estimation schemes. The fractional Onrstein-Uhlenbeck (fOU) process satisfies the above-mentioned condition from \cite{istas}; see \cite{cheridito} for an exhaustive definition and initial study. We also give a brief review of its definition in section \ref{sec:introductionfOU}. 

Kleptsyna and Le Breton \cite{klebreton} proved that a maximum likelihood
estimator of the drift parameter $\lambda$ for the fOU process (in the
non-stationary case) can be obtained using similar techniques to the case
where $H=\frac 1 2$. This is attained by considering $H\in [\frac 1 2,1)$. The
main difference between \cite{klebreton} and \cite{lipster} is the presence of
the fractional kernels within the likelihood-estimate formulas, due to the
generalization of the Girsanov formula to the fBm case. They use this formula
to obtain an MLE of $\lambda$ which is strongly consistent. Furthermore, they
compute approximations of the bias and MSE of their estimator. Tudor and this
paper's second author \cite{tudorviensfbm} extended the previous result to
fractional SDEs where the drift parameter is not necessarily linear. They
proved strong consistency of the MLE estimator even in the case $H<\frac 1 2$,
using Malliavin calculus techniques, together with bias and MSE approximations
for their estimate of $\lambda$. They also provide a discretized estimator
which is still strongly consistent under certain conditions on the drift. A
more recent estimate of $\lambda$ (assuming $\lambda>0$) was developed by Hu
and Nualart \cite{Hu}. They propose a least-squares estimate which is written
in terms of an It\^o stochastic integral, but its numerical manipulation is
difficult when $H>\frac 1 2$; the estimator cannot be implemented directly
when $H<\frac 1 2$ since then, the stochastic integral would have to be
interpreted in the Skorohod sense, 
and the latter cannot be evaluated measurably with respect to the fOU process's
sigma-field, as it would be if the integral were a classical Wiener or
Wiener/Ito integral; in other words, because the Skorohod integral is not a
stochastic integral in the classical sense, but rather is a divergence operator
(it is the adjoint of a derivative operator, and cannot be interpreted as an
antiderivative in the case of fBm-based models), information from the entire
path of the noise term driving of the fOU process would be required, which is
not an explicit functional of the observations of the fOU process itself. The estimate is
strongly consistent for all $H\geq \frac 1 2$ and it turns out that for $H \in
\left[\frac 1 2,\frac 3 4\right)$ the least-squares estimate satisfies a
classical CLT with a variance depending on the unknown parameter
$\lambda$. Moreover, they improved the numerical applicability of the previous
estimate by introducing a new one in terms of a Riemann integral of the
observed process $X_t$ ($H \geq \frac 1 2$), and this estimate turns out to be strongly consistent and satisfies a CLT for $H \in \left[\frac 1 2,\frac 3 4\right)$.
The proofs of strong consistency and the CLT for both estimates rely on
Malliavin calculus techniques, in particular, the Nualart-Pecatti-Ortiz
characterization of normal convergence for multiple stochastic integrals (see
\cite{nualartortiz}). In a more recent study, Brouste and Iacus (see
\cite{Brouste2012}) proved that the variogram of a non-stationary fOU process
satisfies the regularity assumptions of Istas and Lang (see \cite{istas}) and
hence they were able to give a joint estimator of $(H,\sigma^2)$ when
$\lambda$ is fixed. In fact, they proved that it is strongly consistent and
asymptotically normal. When $H\in (\frac 1 2,\frac 3 4)$ they also gave an
estimator of $\lambda$ based on Hu and Nualart's results in \cite{Hu}, under a
certain combination of increasing-domain and infill asymptotics. For $H$ and
$\sigma$ fixed, their estimate is strongly consistent and asymptotically
normal as well. Recently \cite{Es-Sebaiy2015} generalized the previous results
to the estimation of drift and scaling parameters of essentially any
long-memory stationary Gaussian process. That paper was written in response to
their own paper jointly with El Onsy \cite{ElOnsy2014} and several other recent papers cited in \cite{Es-Sebaiy2015} which are other special cases of their framework. However, their method does not appear to handle any memory parameter, nor do any of the methods cited in their work. 

As we mentioned from the outset, our interest is to estimate the parameters in stochastic differential equations using discrete observations, which can be particularly useful in the areas of econometrics and
finance; indeed these fields' long-memory stochastic modeling are typically limited in how frequently data can be observed. Section section
\ref{sec:numer-perf-gmm} contains a brief and specific analysis of such a situation in the case of financial data. Unlike the cases described in the references above, the asymptotics are taken when the
number of observations increases, but the time length among observations remains
fixed. In the case of a fOU process this problem has been studied by \citep{Hu2013} for the
estimation of $\lambda$  when the Hurst parameter is fixed and is
greater than $\frac 1 2$. The joint estimation of the pair $(\lambda,\sigma)$
has been solved for discrete observations in \cite{Xiao2011} when $H$ is kept
fixed in $\left(\frac 1 2,\frac 3 4\right)$. It is important to note that
\cite{Zhang2014} solves the parameter estimation of the entire set of
parameters $(H,\lambda,\sigma)$ for a fixed interval and infill asymptotics, i.e. drawing on availability of data with arbitrarily high sampling frequency. 

In the present article, using the Generalized
Method of Moments, we show how to achieve the joint estimation of any finite vector of parameters for stationary Gaussian sequences at the highest level of generality, using
discrete observations, without requiring infill asymptotics, and we show that our setting applies directly to the joint estimation of $H$, $\lambda$ and
$\sigma$ for the fOU process. Normal asymptotics are established thanks to the Breuer-Major theorem, and comments on efficiency are included. Our work is the first attempt to solve the joint estimation of the fOU process for its three parameters using a classical technique, without using infill asymptotics at all. We show via numerics for the fOU process that our technique is easily implementable, with good empirical precision and robustness properties even when operating with a moderate and realistic dataset size.

An interesting phenomenon occurs in the case of long-memory sequences such as discretely-observed fOU: if using only a Method of Moments based directly on the terms of the sequence itself, not its finite-differences, any estimator based on quadratic variations fails to be asymptotically normal when $H>3/4$. This is a relatively well-known phenomenon, since that range of $H$ falls beyond the Breuer-Major theorem's scope, though typically authors do not seek results in that range. Arguably, such results would be unnecessary since one only needs to consider one order filtered observations to avoid the non-normal asymptotics. However, it is instructive to investigate this case a bit further. Here we provide some calculations for the fOU case showing that the quadratic variation without any filtering cannot converge to a normal law, or even, which is more surprising, to a second-chaos law, but that it remains the basis for a strongly consistent estimator of the variance nonetheless. This requires the use of the Malliavin calculus.

This article is structured as follows: section \ref{sec:Jointgeneral} defines our GMM estimator for
any stationary Gaussian process and states its general properties:
consistency, asymptotic normality and efficiency. Section \ref{sec:param-estim-fract} is an application
of the above results to the joint parameter estimation of the stationary solution of the fractional
Ornstein-Uhlenbeck SDE. Finally, section
\ref{sec:numer-perf-gmm} analyzes the numerical performance of the GMM
estimator under a simulation study.  Appendix \ref{sec:additional-tables} contain additional numerics. Appendix \ref{sec:appendix-proofs} contains mathematical proofs of technical lemmas and other results in the paper. Appendix \ref{sec:behavior-g_0-when} contains the convergence results, with proofs, for the non-finite-differenced estimator for fOU when $H>3/4$.

\section{Joint Estimation of Gaussian Stationary processes}\label{sec:Jointgeneral}
Let $X_t$ be a real-valued centered gaussian stationary process with spectral
density $f_{\theta_0}(x)$, where $f_\theta(x)$ is a continuous function with
respect to $x$, continuously differentiable with respect to $\theta$ ($\theta$
represents a vector of paremeters and it
belongs to a compact set $\Theta \subset \mathbb R^p$) and $\theta_0 \in
\text{Interior}(\Theta)$. Hereafter, $\theta_0$ will be the true parameter of $X_t$. By Bochner's theorem, the autocovariance function of $X_t$ can be written as:
\begin{align}\label{bochner}
  \rho_\theta(s):=\text{Cov}(X_{t+s},X_t)=\int_{\mathbb R}e^{isx}f_\theta(x)dx=\int_{\mathbb R}\cos(sx)f_\theta(x)dx
\end{align}
for $s\geq 0$ and $\theta \in \Theta$. 

Note that if we assume that $\rho_\theta(s)$ is a continuous function of $s$,
then by \cite{grenander} (page 257), the process $X_t$ is ergodic. Take a
positive integer $L\geq p$, and define:
\begin{align}\label{eq:11}
  \nns{\rho}_\theta(\alpha):=(\rho_\theta(\alpha \cdot 0),\ldots,\rho_\theta(\alpha \cdot L))'.
\end{align}
Now we have the condition that let us to identify the parameter $\theta$.
\begin{ap}[Identifiability condition]\label{sec:ap1}
  Take $\alpha>0$. Then the subvector of $\nns{\rho}_\theta(\alpha)$:
  \begin{align*}
    \nns{\rho}_{\theta,p}(\alpha)=(\rho_\theta(\alpha \cdot 0),\ldots,\rho_\theta(\alpha \cdot (p-1)))'.
  \end{align*}
is an injective map in $\theta$. Assume also that the map $\nns{\rho}_{\theta,p}(\alpha)$ is
differentiable. 
\end{ap}
\begin{rk}
  Note that the above assumption can be obtained if $\nabla_\theta
\nns{\rho}_{\theta,p}(\alpha)$ is a non-singular matrix or $\nabla_\theta
\nns{\rho}_\theta(\alpha)$ is a full-column rank matrix for any $\theta \in
\Theta$. Here we are using the notation $\nabla_\theta \cdot$ to mean the
gradient of its argument with respect to $\theta$. 
\end{rk}
The concept of filter will be used along this article. Its definition is as follows:
\begin{defin}\label{deffilters}
  A filter $\nn a=(a_0,\ldots,a_L)$ of length $L+1$ and order $l$ is a sequence of $L+1$ real numbers such that:
  \begin{align*}
    \sum_{q=0}^{L}a_qq^r&=0,\qquad 0\leq r\leq l-1,\; r\in \mathbb{Z}\\
    \sum_{q=0}^{L}a_qq^l&\neq 0
  \end{align*}
when $l>0$. When $l=0$, we will assume that $a_0=1$ and $a_q=0$ for $0<q\leq L$.
\end{defin}
In this article we will employ a family of discrete filters of length $L$ and
orders in  $\{0,\ldots,L\}$. To simplify notation, we define for a filter $a$ with order $l$:
  \begin{align}
    b_k:=
    \begin{cases}\label{eq:bk1}
      \sum_{q=0}^La_q^2 & \text{if $k=0$}\\
      2\sum_{j=0}^{L-k}a_{k+j}a_j & \text{if $k>0$}
    \end{cases}
  \end{align}
and $\nn b:=(b_0,\ldots,b_L)'$. It is straightforward to prove that given two
different orders $l_i$, $l_j$, the corresponding vectors $\nn b_i$ and $\nn b_j$
are linearly independent. For these reason we can assume the following: 
\begin{ap}\label{sec:ap2}
  Assume that we can choose $L$ filters $\nn a_1,\ldots,\nn a_L$ with
  respective orders $l_1,\ldots,l_L$ such that the $L \times L+1$ matrix:
  \begin{align*}
    B=
    \begin{bmatrix}
      \mathbf{b}_1'\\
      \mathbf{b}_2'\\
      \vdots \\
      \mathbf{b}_L'
    \end{bmatrix}
  \end{align*}
satisfies that $p \leq rank(B)\leq L$. 
\end{ap}
Using this last assumption, we can define the function $V(\theta)$, for any $\theta \in \Theta$, as:
\begin{align*}
V(\theta):=\sum_{k=0}^Lb_k\rho_{\theta}(\alpha k),
\end{align*}
and the \textit{filtered process with step size $\alpha>0$} (fixed) at $t\geq 0$ as:
\begin{align*}
\varphi(t):=\sum_{q=0}^La_qX_{t-\alpha q}.
\end{align*}
for a fixed filter $\nn a$.
This filtered process was already considered in previous estimation studies, for example \cite{bierme2008estimation}, \cite{chronopoulou2009variations}, \cite{Coeurjolly2001} and \cite{istas}. Due to the stationarity of $X_t$, we can note that the variable $\varphi(t)$ is centered gaussian, hence the expected value of $\varphi(t)$ is zero and its variance is $E[\varphi(t)^2]=V(\theta_0)$, as follows:
\begin{align*}
  E[\varphi(t)^2]&=\sum_{q,q'=0}^L a_qa_{q'}E[X_{t-\alpha q}X_{t-\alpha q'}]\\
  &=\sum_{q,q'=0}^L a_qa_{q'}\rho_{\theta_0}(\alpha\cdot |q-q'|)\\
  &=\sum_{k=0}^Lb_k\rho_{\theta_0}(\alpha k)=V(\theta_0).
\end{align*}
Following the notation of \cite{mcfadden} we define the vector of differences
among the squared filtered observations for each of the $L$ filters and its
corresponding $V(\theta)$, for any $\theta \in \Theta$ as:
\begin{align*}
\nn g(t,\theta):=(g_1(t,\theta),\ldots,g_L(t,\theta))'
\end{align*}
where
\begin{align*}
g_\ell(t,\theta)=\varphi_\ell(t)^2-V_\ell(\theta),\qquad \text{for $1\leq \ell\leq L$}
\end{align*}
and the subscript $\ell$ means that we use the filter $\mathbf{a}_\ell$ in the
computation of the filtered process and its variance entrywise.
Clearly, in the jargon of Method of Moments' estimation, $\nn g(t,\theta)$ \textit{satisfies a population moment condition} (see \cite{hallgmm} and \cite{mcfadden}): 
\begin{align*}
E[\nn g(t,\theta_0)]=0
\end{align*}
for all $t\geq 0$. Also note that the vector $\nn g$ is a second-Wiener chaos
vector (see \cite{Nourdin2012}).

Let us assume that we have observed the process $X_t$ at times $0=t_0< t_1
<\cdots <t_{N-1}<t_N=T$ and fix $\alpha=t_i-t_{i-1}$. 

Let $A$ be a symmetric positive-definite matrix. We will denote $\|A\|$ as its matrix norm induced by the Euclidean norm in $\mathbb{R}^{L+1}$. It is important to note that $A$ can be chosen to ensure that the GMM estimate is efficient (see section \ref{sec:efficiency}). Denote, for each $\theta \in \Theta$, and for an arbitrary time $t\geq 0$:
\begin{itemize}
\item $\nn g_0(\theta):=E[\nn g(t,\theta)]$ (vector of expected differences)
\item $\hat {\nn g}_N(\theta):=\frac{1}{N-L+1}\sum_{i=L}^N\nn g(t_i,\theta)$
  (vector of sampled differences) 
\item $Q_0(\theta):=\nn g_0(\theta)'A\nn g_0(\theta)$ (squared distance of
  expected differences)
\item $\hat{Q}_N(\theta):=\hat {\nn g}_N(\theta)'A \hat {\nn g}_N(\theta)$
  (sampled version of the above distance)
\end{itemize}
Note that $\nn g_0(\theta)$ does not depend on time $t$, due to the
stationarity assumption on $X_t$. Finally, we define the GMM (Generalized
Method of Moments) estimator of $\theta_0$ as (see \cite{hansen82}, \cite{hallgmm}):
\begin{align}\label{GMMest}
\hat \theta_N=\text{argmin}_{\theta \in \Theta}\hat{Q}_N(\theta).
\end{align}
The main idea behind the GMM estimation is that since the function $Q_0(\theta)$ attains a unique zero at $\theta_0$ under Assumptions \ref{sec:ap1} and \ref{sec:ap2} (see Lemma \ref{unique} below), we would want to find a value $\hat \theta_N$ (that depends on the current observations) such that our weighted empirical approximation to $Q_0$, namely $\hat Q_N$, should be as small as possible.

It is easy to prove that all the remaining results are still valid if we substitute the fixed matrix $A$ with a sequence of random matrices (perhaps depending on the data) with a deterministically bounded eigenstructure. In particular, we can choose such sequence in order to attain convergence in probability to the efficient alternative of $A$. For more details see section \ref{sec:efficiency}. 
\subsection{Consistency}
In this section we will prove the strong consistency of the estimator defined in (\ref{GMMest}). Our first lemma shows that the sequence $\hat Q_N(\theta)$ is uniformly convergent for any $\theta \in \Theta$.
\begin{lema}\label{lemaunif1}
  It holds that:
  \begin{align*}
    \sup_{\theta \in \Theta}|\hat Q_N(\theta)-Q_0(\theta)|\stackrel{a.s}{\rightarrow}0.      
  \end{align*}
\end{lema}
\begin{proof}
  See Appendix \ref{sec:appendix-proofs}.
\end{proof}
Under the injectivity assumptions on $\nns \rho_\theta(\alpha)$, we can prove:
\begin{lema}\label{unique}
  Under Assumptions \ref{sec:ap1} and \ref{sec:ap2}:
  \begin{align*}
    Q_0(\theta)=0.
  \end{align*}
if and only if $\theta=\theta_0$.
\end{lema}
\begin{proof}
  See Appendix \ref{sec:appendix-proofs}.
\end{proof}
And these two lemmas allow us to prove:
\begin{teo}[Strong Consistency]\label{teoconsis1}
Under Assumptions \ref{sec:ap1} and \ref{sec:ap2}, it holds that:
\begin{align*}
\hat \theta_N\stackrel{a.s.}{\rightarrow}\theta_0.  
\end{align*}
\end{teo}
\begin{proof}
Because of Lemma \ref{lemaunif1} and \ref{unique} we can apply Theorem 2.1 in \cite{mcfadden} and the result holds.
\end{proof}
\subsection{Asymptotic Normality}\label{sec:asymptotic-normality-section}
In this section we use the Breuer-Major theorem (see \cite{breuermajor}) to prove the asymptotic normality of $\hat \theta_N$, under some specific assumptions on $\nns \rho_\theta(\alpha)$ and $f_\theta(x)$. First denote for any $\theta \in \Theta$:
\begin{align*}
  \hat G_N(\theta)&:=\nabla_\theta \hat {\nn g}_N(\theta)\\
  G(\theta)&:=E[\nabla_\theta \nn g(\cdot,\theta)].
\end{align*}
Note that $G(\theta)$ does not depend on $t$ because of the
stationarity of $X_t$. 
Now we can state the following lemma:
\begin{lema}\label{asymnorm1}
Let $\theta \in \Theta$. The following is true:
\begin{itemize}
\item[(i)] $\hat G_N(\theta)=G(\theta)=-\nabla_\theta \mathbf{V}(\theta)$. Moreover, $G(\theta)$ is a continuous function of $\theta$.
\item[(ii)] Under Assumptions \ref{sec:ap1} and \ref{sec:ap2} the matrix $G(\theta)'DG(\theta)$ is non-singular for any positive-definite matrix $D$. 
\end{itemize}
\end{lema}
\begin{proof}
See Appendix \ref{sec:appendix-proofs}.
\end{proof}

The fact that $X_t$ is a stationary process allows us to write this gaussian process as a Wiener-It\^o integral with respect to a complex centered Gaussian random measure $W$ over $[-\frac{\pi}{\alpha}, \frac{\pi}{\alpha}]$ as follows (see Section 3.2 in \cite{Lifshits2012}, Theorem 2.7.7 in \cite{Nourdin2012} and Section 2.1 in \cite{Bierme2011}),
\begin{align}\label{eq:9}
  X_{t_k}=\int_{-\frac{\pi}{\alpha}}^{\frac{\pi}{\alpha}} e^{it_kx}\bar f_{\theta_0}(x)^{1/2}dW(x)=I_1(A_k(\cdot|\theta_0))
\end{align}
where $A_k(x |\theta_0):=e^{it_kx}\bar f_{\theta_0}(x)^{1/2}$ is the
corresponding kernel under this representation for $k\in \{0,\ldots,N\}$ and
$\bar f_{\theta_0}(x)$ is the spectral density of $X_t$ (for $t_k=k\cdot \alpha$, $k\in
\mathbb Z$):
\begin{align}\label{eq:7}
  \bar{f}_{\theta_0}(u):=\sum_{p\in \mathbb
    Z}f_{\theta_0}\left(u+\frac{2\pi}{\alpha} p\right), \qquad u\in \left[-\frac{\pi}{\alpha}, \frac{\pi}{\alpha}\right].  
\end{align}
Assume that we choose $n$ and $m$ such that $t_n<t_m$. The covariance between
$X_{t_m}$ and $X_{t_n}$ can be computed as: 
\begin{align}\label{eq:10}
  \text{Cov}(X_{t_m},X_{t_n})&=\rho_{\theta_0}(|t_m-t_n|)\notag\\
&=\frac{\alpha}{2\pi}\int_{-\frac{\pi}{\alpha}}^{\frac{\pi}{\alpha}}e^{i|m-n|\alpha
  x}\bar f_{\theta_0}(x)dx\notag \\
&=\frac{1}{2\pi}\int_{-\pi}^\pi e^{i|m-n|u}\bar{f}_{\theta_0}\left(\frac u \alpha\right)du.
\end{align}
For the next lemma, we need to generalize the definition of $b_k$ in (\ref{eq:bk1}) in the following way:
  \begin{align*}    b_k^{i,j}:=\sum_{q=0}^{L-k}a_q^ia_{q+k}^j+\sum_{q=0}^{L-k}a_q^ja_{q+k}^i
  \end{align*}
where $i,j=1,\ldots,L$ and $a_k^i$ is the $k$-th entry of the filter
$\mathbf{a}_i$ (recall the notation in Assumption \ref{sec:ap2}). The main
condition in the Breuer-Major theorem is the summability of the second-power
of the autocovariance function at integer times (see \cite{breuermajor}). In
order to validate this argument we can make the following assumption, which is
equivalent to the above fact:
\begin{ap}\label{sec:ap4}
Assume that for any $i,j\in \{1,\ldots,L\}$ we have that:
  \begin{align*}
    R(u|i,j):=\min\left(1,|u|^{l_i+l_j}\right)\bar{f}_{\theta_0}(u) \in L^2[(-\pi,\pi)],
  \end{align*} 
where $l_i$ is the order of the filter $\mathbf{a}_i$. 
\end{ap}

The assumption above and the fact that the vector $\hat {\nn g}_N(\theta_0)$ can be rescaled to obtain a sum of Hermite processes of second order, are sufficient conditions to prove the asymptotic normality of the sample moment equations, as follows:
\begin{lema}\label{sec:asymptotic-normality}
Under Assumptions \ref{sec:ap1}-\ref{sec:ap4}:
  \begin{align*}
    \sqrt{N}\hat {\nn g}_N(\theta_0)\stackrel{d}{\rightarrow}\mathcal{N}(0,\Omega)
  \end{align*}
  where $\Omega$ is a $L\times L$ matrix whose components are:
  \begin{align}\label{eq:1}
    \Omega_{ij}=2\sum_{k\in \mathbb Z}\left[\sum_{p=0}^Lb_p^{i,j}\rho_{\theta_0}\left[\alpha(k+p)\right]\right]^2.
  \end{align}
\end{lema}
\begin{proof}
See Appendix \ref{sec:appendix-proofs}.
\end{proof}

The next natural step in this analysis is to study the asymptotic behavior of the Mean Squared error (MSE) of $\hat \theta_N$. Note that from equation (\ref{eq:9}) the filtered process can be represented for any $\ell \in\{1,\ldots,L\}$ and $i \in \{0,1,\ldots,N\}$ as:
\begin{align}\label{eqphiell}
  \varphi_\ell(t_i)=\sum_{q=0}^La_qX_{t_i-\alpha q}=I_1(SA_{i,\ell}(\cdot|\theta_0))
\end{align}
where $SA_{i,\ell}(u|\theta_0):=\sum_{q=0}^La_qA_{i-q}(u|\theta_0)$, $u\in \mathbb R$. Hence, using the multiplication rule of Wiener integrals (see \cite{nualartbook}), the componentwise average of the sample moment equations is:
\begin{align*}
  (\hat {\nn g}_N(\theta))_\ell&=\frac{1}{N-L+1}\sum_{j=L}^N [\varphi_\ell(t_j)^2-V_\ell(\theta)]\\
&=\frac{1}{N-L+1}\sum_{j=L}^NI_2(SA_{\ell,j}\otimes SA_{\ell,j})=I_2[B_{\ell,j}(\cdot|\theta)]
\end{align*}
where $B_{\ell,j}(\cdot|\theta):=\frac{1}{N-L+1}\sum_{j=L}^NSA_{\ell,j}\otimes SA_{\ell,j}$ and $\ell \in \{1,\ldots,L\}$. Then we can conclude that $\hat{\nn g}_N(\theta)$ is a vector of $L$ integrals belonging to the second Wiener-chaos of $W$. Because of Lemma \ref{sec:asymptotic-normality}, we already know that:
\begin{align*}
  \sqrt{N}(\hat {\nn g}_N(\theta_0))_\ell\stackrel{d}{\longrightarrow}\mathcal{N}(0,\Omega_{\ell,\ell})
\end{align*}
and therefore $E[|(\hat {\nn g}_N(\theta_0))_\ell|^2]=O(N^{-1})$. Hence there exists a constant $K_\ell>0$ such that:
\begin{align*}
  E[|(\hat {\nn g}_N(\theta_0))_\ell|^2]\leq \frac{K_\ell}{N}.
\end{align*}
so, we can use the fact that $(\hat {\nn g}_N(\theta_0))_\ell$ belongs to the
second-Wiener chaos generated by the process $W(x)$ and furthermore we can use
the equivalence of all the $L^s$-norms of $(\hat {\nn g}_N(\theta_0))_\ell$ to get: (see Theorem 3.50 in \cite{janson1997gaussian})
\begin{align*}
  E[|(\hat {\nn g}_N(\theta_0))_\ell|^{4s}]&\leq c_s[E[|(\hat {\nn g}_N(\theta_0))_\ell|^2]]^{2s}\leq \frac{c_s\cdot K_\ell}{N^{2s}}
\end{align*}
where $c_s>0$. Using Jensen's inequality:
\begin{align*}
  E[\|\hat {\nn
    g}_N(\theta_0)\|^{4s}]&=E\left[\left|\sum_{\ell=1}^L\left|(\hat {\nn
          g}_N(\theta_0))_\ell\right|^2\right|^{2s}\right]\leq
  L^{2s-1}\sum_{\ell=1}^LE\left[|(\hat {\nn g}_N(\theta_0))_\ell|^{4s}\right]\\
&\leq \frac{L^{2s}c_s}{N^{2s}}\max_{1\leq \ell\leq L}\{K_\ell\}.
\end{align*}
All these calculations allow us to study the behavior of the MSE of $\hat
\theta_N$. Let $\varepsilon_N:=\hat \theta_N-\theta_0$. First note that
because of the multivariate mean value theorem (see Proposition F.6.1 in
\cite{Lebanon2012} and more details in \cite{mcfadden}):
\begin{align*}
  \varepsilon_N:=\hat \theta_N-\theta_0&=-[\hat G_N(\hat \theta_N)'A \hat G_N(\bar \theta_N)]^{-1}\cdot \hat G_N(\hat \theta_N)'A\hat g_N(\theta_0)\\
  &=-\psi_N(\bar \theta_N,\hat \theta_N)\cdot \hat g_N(\theta_0)
\end{align*}
where $\bar \theta_N$ is a value in the interval $[\theta_0,\hat \theta_N]$, 
\begin{align}\label{eq:6}
  \psi_N(\bar \theta_N,\hat \theta_N):=[G(\hat \theta_N)'A G(\bar \theta_N)]^{-1}\cdot G(\hat \theta_N)'A
\end{align}
and $\theta \in \Theta$ (see Lemma \ref{asymnorm1}). The next lemma states
that all moments of $\psi_N$ can be uniformly bounded:
\begin{lema}\label{sec:LemacotaEPsi}
  Let $\psi_N$ as in equation (\ref{eq:6}). Under the remark after assumption \ref{sec:ap1} there exists $R_{s,L}>0$ such that:
  \begin{align*}
    E[\|\psi_N(\bar \theta_N,\hat \theta_N)\|^{4s}]<R_{s,L}.
  \end{align*}
\end{lema}
\begin{proof}
See Appendix \ref{sec:appendix-proofs}.
\end{proof}

Note also that because of the continuity of $G$:
\begin{align*}
  \psi_N(\bar \theta_N,\hat \theta_N)\stackrel{a.s.}{\rightarrow}[G(\theta_0)'AG(\theta_0)]^{-1}G(\theta_0)'A.
\end{align*}
and due to Lemma \ref{sec:LemacotaEPsi}, for any $s>0$:
\begin{align*}
  E[\|\varepsilon_N\|^{2s}]&=E\left[\|\psi_N(\bar \theta_N,\hat \theta_N)\cdot \hat {\nn g}_N(\theta_0)\|^{2s}\right]=E\left[\|\psi_N(\bar \theta_N,\hat \theta_N)\|^{2s}\cdot \|\hat {\nn g}_N(\theta_0)\|^{2s}\right]\\
  &\leq E\left[\|\psi_N(\bar \theta_N,\hat \theta_N)\|^{4s} \right]^{1/2}E\left[\|\hat {\nn g}_N(\theta_0)\|^{4s}\right]^{1/2}<\frac{\mathcal R_{s,L}}{N^s} 
\end{align*}
where $\mathcal R_{s,L}=(L\cdot R_{s,L}\cdot c_s\cdot \max_{\ell}\{K_\ell\})^{1/2}$, and taking $s=1$ we obtain an upper bound for the MSE of $\theta_N$. Furthermore, if we choose $s>1$, we can apply Borel-Cantelli lemma to get:
\begin{align*}
  \varepsilon_N\stackrel{\text{a.s}}{\longrightarrow}0
\end{align*}
If we take $\gamma>0$, and under the same reasoning as before, we can use Chebyshev's inequality to get:
\begin{align*}
  P[\|\varepsilon_N\|>N^{-\gamma}]\leq \frac{E[\|\varepsilon_N\|^{2s}]}{N^{-2\gamma s}}<\frac{\mathcal R_{s,L}}{N^{(1-2\gamma)s}}
\end{align*}
and hence for any $\gamma<\frac{1}{2}$ there exists $s$ such that $s>\frac{1}{1-2\gamma}$ and, by Borel-Cantelli lemma there exists $N_0(\omega)<\infty$ such that for all $N>N_0(\omega)$:
\begin{align*}
  N^{\gamma}\|\varepsilon_N\|<c \qquad \text{a.s.}
\end{align*}
 where $c$ is an arbitrary positive constant. Hence, we conclude the following theorem:
 \begin{teo}\label{sec:MSEteoGMM}
   Under Assumptions \ref{sec:ap1} and \ref{sec:ap2} we have:
   \begin{itemize}
   \item[(i)] $E[\|\hat \theta_N-\theta_0\|^2]=O(N^{-1})$
   \item[(ii)] $N^\gamma\|\hat \theta_N-\theta_0\|\stackrel{\text{a.s.}}{\longrightarrow}0$ for any $\gamma<\frac 1 2$. 
   \end{itemize}
 \end{teo}
Now we can state a CLT for $\hat \theta_N$, under Assumptions \ref{sec:ap1}-\ref{sec:ap4}:
\begin{teo}[Asymptotic Normality]\label{sec:asympnorm}
Let $X_t$ be a Gaussian stationary process with parameter $\theta_0$. If $\hat \theta_N$ is the GMM estimator given in (\ref{GMMest}), then under Assumptions \ref{sec:ap1}-\ref{sec:ap4}, it holds that:
\begin{align*}
  \sqrt{N}(\hat \theta_N-\theta_0)\stackrel{d}{\rightarrow}\mathcal{N}(0,\mathcal C(\theta_0)\Omega \mathcal C(\theta_0)')
\end{align*}
where $\mathcal C(\theta_0)=\left[G(\theta_0)'AG(\theta_0)\right]^{-1}G(\theta_0)'A$ and $\Omega$ is defined as in Lemma \ref{sec:asymptotic-normality}.
\end{teo}
\begin{proof}
  Lemmas \ref{asymnorm1} and \ref{sec:asymptotic-normality} gives sufficient conditions to apply Theorem 3.4 in \cite{mcfadden}, and conclude the asymptotic normality of $\hat \theta_N$.
\end{proof}
\subsection{Practical Considerations}\label{sec:efficiency}
If the number of moment conditions $L$ is greater than the number of parameters $p$, the asymptotic variance of $\sqrt{N}(\hat \theta_N-\theta_0)$ is minimized when $A=\Omega^{-1}$ (see Theorem 3.4 of \cite{hallgmm} and \cite{mcfadden} for a detailed proof of this fact). In the case of a number of moment equations $L$ equal to the number of parameters $p$, the estimation problem is equivalent to the Method of Moments and hence $A=I$.

There are some numerical methods to compute a sequence of positive-definite
matrices such that $\hat A_N\stackrel{d}{\longrightarrow}A$ in such a way that
$\hat A_N$ depends on the data available. This type of procedure gives
numerical advantages, specifically when $A=\Omega^{-1}$, since the previous matrix
depends on the unknown parameter $\theta$, while the sequence $\hat A_N$ do
not depend on this parameter. Three methods that are commonly used to achieve the above are:
\begin{enumerate}
\item Hansen's Two-Step estimator (see \cite{hansen82}),
\item Iterative estimator and 
\item Continuous-Updating estimator (CUE) of Hansen, Heaton and Yaron (see \cite{hansen1996finite}).
\end{enumerate}
For ease in our calculations in sections \ref{sec:param-estim-fract} and \ref{sec:numer-perf-gmm}, we decided to use a constant matrix $A=\Omega^{-1}$ under
the case where $L>p$ and $A=I$ otherwise.   

\section{Parameter Estimation of the fractional Ornstein-Uhlenbeck
  process}\label{sec:param-estim-fract}

The main objective of this section is to apply the GMM method in the parameter estimation of a particular gaussian stationary process: the fractional Ornstein-Uhlenbeck process (fOU), under an increasing domain approach. First we will give a brief introduction of the fOU process.

\subsection{Introduction}\label{sec:introductionfOU}
Cheridito \cite{cheridito} defines the fractional Ornstein-Uhlenbeck process with initial condition $\xi\in L^0(\Omega)$, as the unique almost surely solution of the Langevin equation:
\begin{align*}
  X_t=\xi-\lambda \int_0^tX_sds+\sigma B_t^H,\qquad t\geq 0,
\end{align*}
where $\lambda, \sigma>0$, $H\in (0,1]$ and we use the notation $\theta=(H,\lambda,\sigma)$. The stationary solution of the above equation is written as:
\begin{align}\label{eq:4}
  X_t=\sigma\int_{-\infty}^te^{-\lambda(t-u)}dB_u^H,\qquad t\in \mathbb R,
\end{align}
where the integral is understood in the sense of Riemann-Stieltjes. Since this process is centered-gaussian and stationary, we can describe its distributional properties by using only its autocovariance function:  
\begin{align}\label{covcompleto}
\rho_\theta(t)=2\sigma^2c_H\int_0^{\infty}\cos(tx)\frac{x^{1-2H}}{\lambda^2+x^2}dx
\end{align}
where $H \in (0,1)$ and $c_H=\frac{\Gamma(2H+1)\sin(\pi H)}{2\pi}$. Note that
when $H\geq \frac 1 2$ we can obtain an alternative expression for
$\rho_\theta(t)$ (see Lemma \ref{lemacovfOU} and formula \eqref{covmedio} in
the Appendix \ref{sec:appendix-proofs}). Note that using equation \eqref{covcompleto}, we can write an explicit expression of the spectral density of $X_t$ in continuous time:
\begin{align*}
  f_\theta(x)=\sigma^2c_H\frac{x^{1-2H}}{x^2+\lambda^2}.
\end{align*}

Throughout this section we will assume that there exists a closed rectangle $\Theta\subset \mathbb R^2$ such that $\theta\in \Theta$.

\subsection{Consistency}\label{sec:fOUconsistency}
For ease of computation we will study the joint estimation of the three
parameters in $\theta$ when $H\geq \frac 1 2$, and we will employ the
autocovariance function obtained in formula \eqref{covmedio} only. Note that
$f_\theta(x)$ is continuous with respect to $x$, $H$, $\lambda$ and $\sigma$ and its partial derivatives are also continuous functions of its parameters.

Assumption \ref{sec:ap1} is verified locally because of
Lemma \ref{sec:lemainjectivity} in Appendix \ref{sec:appendix-proofs}. Note that the sufficient conditions in order to the
above assumption to hold are (1) $\alpha<1$, (2) $\sigma$ is of known sign and
(3) $0<\lambda<\exp(\Psi(3))$. The condition (2) does not change the magnitude of the variance of the process $X_t$
nor its scale, since the process is centered. Finally, the joint GMM estimator defined in (\ref{GMMest}) is consistent under the previous limitations. 

The ergodicity of $X_t$ is deduced from the expansion of the function $\rho_\theta(x)$ when $H\neq \frac 1 2$ (see \cite{cheridito}):
\begin{align}\label{expancorr}
  \rho_\theta(x)=\frac{H(2H-1)}{\lambda}x^{2H-2}+O(x^{2H-4}),\qquad \text{as $x\rightarrow \infty$}.
\end{align}
 
\subsection{Asymptotic Normality}\label{sec:asymptotic-normality-1}
Assume that we have a collection of at least 3 filters satisfying Assumption \ref{sec:ap2}. Hence, we can study under which cases Assumption \ref{sec:ap4} is satisfied:
\begin{lema}\label{sec:asymptotic-normality-2}
  Let $X_t$ be the stationary fOU process with parameters $\theta=(H,\lambda,\sigma)$. The Assumption \ref{sec:ap4} holds under the following two cases:
  \begin{description}
  \item[Case 1] If $l+l'\geq 1$ then it holds for all $H \in (0,1)$,
  \item[Case 2] If $l+l'=0$ then it holds if $H \in (0,\frac 3 4)$,
  \end{description}
where $L\geq 3$.
  \begin{proof}
    See Appendix \ref{sec:appendix-proofs}.
  \end{proof}
\end{lema}

Since the case 1 in the above lemma guarantees the asymptotic normality of the
moment equations for any $H\in (0,1)$, we decided to use filters with
order greater or equal to 1. If this is the case we can use the previous lemma
together with Lemma \ref{sec:asymptotic-normality} to obtain the asymptotics
of $\hat {\nn g}_N(\theta)$ as follows:
\begin{align}\label{eq:ghatstats}
  \sqrt{N}\hat {\nn g}_N(\theta)\stackrel{d}{\longrightarrow}N(0,\Omega(\theta))
\end{align}
where the covariance matrix $\Omega(\theta)$ has entries according to equation (\ref{eq:1}).  

The behavior of 
\begin{align*}
  (\hat {\nn g}_N(\theta))_0:=\frac{1}{N-L+1}\sum_{i=1}^N\varphi_0(t_i)^2-V_0(\theta)
\end{align*}
deserves interest by itself due to its asymptotic behavior when $H\geq \frac 3
4$ and whose proofs rely heavily on Malliavin calculus techniques. We summarized these results in Appendix \ref{sec:behavior-g_0-when}.

We finish this section by applying the general results of the previous section to obtain consistency and asymptotic normality for the fOU's GMM estimator. Assumption \ref{sec:ap1}  on injectivity of the estimator is satsified, as shown in Lemma \ref{sec:lemainjectivity} in Appendix \ref{sec:appendix-proofs}; this lemma establishes this assumption for the fOU observations under an upper bound on $\lambda$ and for $\alpha$ sufficiently small, where these conditions do not depend on $N$. First, we can use the general result in Theorem \ref{sec:MSEteoGMM} to conclude the following rates of convergence of the GMM estimator:
\begin{prop}\label{sec:fract-ornst-uhlenb-boundsMSE}
  Let $X_t$ be a fOU process with parameters
  $\theta=(H,\lambda,\sigma)$. Then, for any positive-definite matrix A and
  under the conditions of Lemma \ref{sec:lemainjectivity}, the GMM estimator $\hat \theta_N$ satisfies:
  \begin{itemize}
  \item[(i)]
    \begin{align*}
      E[\|\hat \theta_N-\theta\|^2]=O(N^{-1})
    \end{align*}
  \item[(ii)] As $N$ goes to infinity:
    \begin{align*}
      N^{\gamma}\|\hat \theta_N-\theta\|\stackrel{\text{a.s.}}{\longrightarrow}0
    \end{align*}
    for all $\gamma<\frac 1 2$.
  \end{itemize}
\end{prop}

Now we are ready to state the asymptotic distribution of our joint estimator, using the limiting behavior of $\hat {\nn g}_N$:
\begin{prop}\label{sec:fract-ornst-uhlenb}
  Let $X_t$ be the stationary fOU process with parameters
  $\theta=(H,\lambda,\sigma)$. Then, for any positive-definite matrix $A$ and
  under the conditions of Lemma \ref{sec:lemainjectivity}, the
  GMM estimator $\hat \theta_N$ of $\theta$ in equation (\ref{GMMest}) is consistent for any
  $H \in (0,1)$ and:
    \begin{align}\label{eq:5}
      \sqrt{N}(\hat \theta_N-\theta)\stackrel{d}{\rightarrow}N(0,\mathcal
      C(\theta)\Omega \,\mathcal C(\theta)')
    \end{align}
    where $\mathcal C(\theta)=[G(\theta)'AG(\theta)]^{-1}G(\theta)'A$ and
    $\Omega$ according to equation (\ref{eq:1}).
\begin{proof}
  The consistency is immediate from Theorem \ref{teoconsis1}, since $X_t$
  satisfies Assumption \ref{sec:ap1} locally. The asymptotic normality is
  attained since assumption \ref{sec:ap4} holds for any pair of filter lengths
  whose sum is greater or equal to 1, which is the case in this
  example. Therefore, we can apply Theorem \ref{sec:asympnorm} to get the desired result.
\end{proof}
\end{prop}
As we noted in section \ref{sec:efficiency}, we chose $A=\Omega^{-1}$ because
this ensures that the asymptotic variance is minimal. If this is the case the
asymptotic variance can be simplified to:
\begin{align*}
  [G(\theta)'\Omega^{-1}G(\theta)]^{-1}.
\end{align*}

\section{Numerical Performance of the GMM estimators for the fractional OU process}\label{sec:numer-perf-gmm}

In this section we test empirically the properties of the GMM estimator of
$(H,\lambda,\sigma)$ for the fOU case. We are especially interested in
checking the performance of the joint estimator for a fixed value of $N$, particularly when $N$ is not sufficiently large that the constants coming from the asymptotic variances become secondary concerns. As noted below, we run tests for $N=1000$ and $N=600$ observation times of a single path of the stochastic process, which represents a realistic dataset size for such questions as the estmation of memory length in financial markets under stochastic volatility (see \cite{Chronopoulou2012}) or in late-Holocene paleoclimatology (See \cite{Barboza2014}), where single observation paths at discrete times are all that is available. These are examples of situations which are not data-rich, but where the amount of data should be sufficient to provide evidence of long memory. Specifically in the case of financial markets with stochastic volatility memory, evidence exists that such memory does not persist for longer than a few weeks at a given level. Assuming for instance that one has access to 20 working days of stock-, option-, and/or volatility-market data, which corresponds to calibrating models to one-month options, in order to stay in the realm of continuous-time models, i.e. to avoid having to consider jump models to account for market microstructure, the best one typically hopes for (liberal estimate) is to observe prices every 5 to 15 minutes depending on option liquidity. This would correspond to roughly 600 to 2000 datapoints before one would need to recalibrate or change the model, hence our order of magnitudes for $N$.  This section provides empirical estimates of
the bias and variance of our joint estimator on all three parameters, and a comparison to the the joint estimation of the pair $(H,\sigma)$ when $\lambda$ is fixed. 

First we choose a fixed step size among observations ($\alpha$). In order to
simulate the fOU process we employ the formula \eqref{covmedio}
together with the Choleski
decomposition of the autocovariance matrix of this gaussian process. This
procedure gives us $m$ independently-generated paths of the fOU process, each such path, sampled at the $N=1000$ observation times, representing one replication of the single-trajectory dataset available in practice. In our simulations we use $m=1000$ replications. We decided to use
finite-difference filters since they are the most typical $p$-vanishing-moments
filters in the literature (see \cite{Coeurjolly2001,
  chronopoulou2009variations,Brouste2012} for example). The filters were selected
sequentially, i.e. we choose $l_i=i$ for $i=1,\ldots,L$, for the purpose of
ensuring that the respective matrix $B$ satisfies Assumption
\ref{sec:ap2} in page \pageref{sec:ap2}.
   
In order to solve numerically for $\hat \theta_N=(\hat H_N,\hat \lambda_N,\hat \sigma_N)$ in
(\ref{GMMest}), we use the quasi-Newton L-BFGS-R algorithm contained in the
R-routine \verb|optim| (see \cite{byrd1995limited} and
\cite{nocedaloptim}). This optimization algorithm is designed for solving
non-linear problems with simple bounds. In order to check that there was an acceptable level of variability produced by the optimization method, 
we compared the theoretical bounds induced in formula \eqref{eq:5} with the empirical variance $e(Var)$ defined the largest eigenvalue of the empirical covariance matrix of the estimator vector. 99\% of the confidence regions constructed with the empirical variance contained the true parameter for each scenario. 

Finally, we ran different scenarios for the purpose of verification of our
GMM procedure. We chose four different values of the parameter $H$
(0.55, 0.65, 0.75 and 0.85) so that it is possible to assess the effect of
increasing levels of long-memory in our estimates' performance. In addition,
for each of the above cases we worked with subscenarios with an increase of
10\% in $\lambda$ and a decrease of $10\%$ in the value of $\sigma$, 
as a quick check on the robustness the empirical statistics. We noted no significant aberrant behavior in these various scenarios; though all results are reported below, readers may safely concentrate on the case $\lambda =\sigma =1$. We chose a maximum number of filters equal to 7, and since there are
three parameters to estimate, we selected an initial order equal to 3. The
scenarios were chosen so that we could analyze the impact of small changes on
the actual values of the parameters in a given fOU process.  

Tables \ref{tab:uno} and \ref{tab:dos} comprise three empirical meatrics of comparison among
scenarios: 
\begin{align*}
  \widehat{MSE}&:=\frac{1}{m}\sum_{i=1}^m\|\hat \theta_{N,i}-\theta\|^2\\
  e(\widehat{Var})&:=\text{maximum eigenvalue of}\;
  \widehat{Var}(\theta_N)\\
\widehat{Bias}^2&:=\left\|\frac{1}{m}\sum_{i=1}^m\hat \theta_{N,i}
  -\theta\right \|^2.
\end{align*}
where $\hat \theta_{N,i}$ is the GMM estimate of $\theta$ using a sample size
of $N$ which was obtained using the $i$-th fOU
replication. $\widehat{Var}(\theta_N)$ is the empirical covariance matrix of
the GMM estimator based on the $m$ replications. These metrics are computed as the number of filters increases
(i.e. as $L$ increases). The calculations of Table \ref{tab:uno} use a sample
size $N=1000$ and table \ref{tab:dos} use $N=600$. The second table also uses an increased time-step of $\alpha=0.5$ which in principle might affect estimator accuracy, though this seems to also depend on the value of $H$. We have also computed the empirical statistics when $\lambda$ is assumed to be known, and when using another class of filters, with results shown in Appendix \ref{sec:additional-tables}. Below are some conclusions:

\begin{table}
\caption{Comparison of Scenarios ($\alpha=0.1$).}
\begin{center}
\begin{tabular}{cc|ccc|ccc|ccc}
\toprule
\multicolumn{2}{c|}{} & \multicolumn{3}{c|}{$\lambda=1$, $\sigma=1$} &
\multicolumn{3}{c|}{$\lambda=1.1$, $\sigma=1$} &
\multicolumn{3}{c}{$\lambda=1$, $\sigma=0.9$}\\
\midrule
$H$& $L$&$\widehat{MSE}$& $e(\widehat{Var})$ &$\widehat{Bias}^2$& $\widehat{MSE}$&
$e(\widehat{Var})$ &$\widehat{Bias}^2$ & $\widehat{MSE}$& $e(\widehat{Var})$
&$\widehat{Bias}^2$\tabularnewline
\midrule
\multirow{5}{*}{0.25}&3&$0.058$&$0.057$&$0.000$&$0.054$&$0.053$&$0.000$&$0.048$&$0.047$&$0.000$\tabularnewline
&4&$0.049$&$0.047$&$0.000$&$0.048$&$0.047$&$0.000$&$0.042$&$0.041$&$0.000$\tabularnewline
&5&$0.045$&$0.043$&$0.000$&$0.045$&$0.043$&$0.000$&$0.040$&$0.039$&$0.000$\tabularnewline
&6&$0.043$&$0.042$&$0.000$&$0.045$&$0.043$&$0.000$&$0.038$&$0.037$&$0.000$\tabularnewline
&7&$0.043$&$0.042$&$0.000$&$0.045$&$0.043$&$0.000$&$0.038$&$0.037$&$0.000$\tabularnewline
\midrule
\multirow{5}{*}{0.40}&3&$0.065$&$0.062$&$0.001$&$0.082$&$0.078$&$0.002$&$0.066$&$0.063$&$0.001$\tabularnewline
&4&$0.058$&$0.056$&$0.001$&$0.074$&$0.070$&$0.002$&$0.062$&$0.060$&$0.001$\tabularnewline
&5&$0.053$&$0.051$&$0.000$&$0.067$&$0.064$&$0.002$&$0.058$&$0.056$&$0.001$\tabularnewline
&6&$0.052$&$0.049$&$0.000$&$0.068$&$0.064$&$0.002$&$0.056$&$0.054$&$0.001$\tabularnewline
&7&$0.052$&$0.049$&$0.000$&$0.068$&$0.064$&$0.002$&$0.057$&$0.054$&$0.001$\tabularnewline
\midrule
\multirow{5}{*}{0.55}&3&$0.061$&$0.056$&$0.002$&$0.080$&$0.074$&$0.003$&$0.062$&$0.059$&$0.002$\tabularnewline
&4&$0.060$&$0.055$&$0.002$&$0.079$&$0.074$&$0.003$&$0.063$&$0.059$&$0.002$\tabularnewline
&5&$0.060$&$0.055$&$0.002$&$0.080$&$0.075$&$0.003$&$0.064$&$0.060$&$0.002$\tabularnewline
&6&$0.061$&$0.055$&$0.002$&$0.081$&$0.075$&$0.003$&$0.063$&$0.059$&$0.002$\tabularnewline
&7&$0.061$&$0.055$&$0.002$&$0.081$&$0.075$&$0.003$&$0.064$&$0.060$&$0.002$\tabularnewline
\midrule
\multirow{5}{*}{0.70}&3&$0.115$&$0.103$&$0.008$&$0.136$&$0.124$&$0.007$&$0.112$&$0.102$&$0.006$\tabularnewline
&4&$0.115$&$0.103$&$0.007$&$0.130$&$0.119$&$0.007$&$0.105$&$0.096$&$0.005$\tabularnewline
&5&$0.109$&$0.099$&$0.006$&$0.126$&$0.116$&$0.006$&$0.102$&$0.094$&$0.004$\tabularnewline
&6&$0.104$&$0.095$&$0.005$&$0.120$&$0.111$&$0.005$&$0.099$&$0.091$&$0.004$\tabularnewline
&7&$0.106$&$0.097$&$0.005$&$0.118$&$0.109$&$0.005$&$0.096$&$0.088$&$0.004$\tabularnewline
\midrule
\multirow{5}{*}{0.85}&3&$0.258$&$0.203$&$0.042$&$0.248$&$0.202$&$0.034$&$0.233$&$0.191$&$0.031$\tabularnewline
&4&$0.219$&$0.197$&$0.009$&$0.244$&$0.225$&$0.008$&$0.224$&$0.205$&$0.009$\tabularnewline
&5&$0.179$&$0.162$&$0.005$&$0.205$&$0.189$&$0.005$&$0.185$&$0.171$&$0.004$\tabularnewline
&6&$0.160$&$0.145$&$0.004$&$0.177$&$0.163$&$0.003$&$0.163$&$0.150$&$0.003$\tabularnewline
&7&$0.156$&$0.141$&$0.004$&$0.169$&$0.155$&$0.003$&$0.154$&$0.141$&$0.003$\tabularnewline
\bottomrule
\end{tabular}
\end{center}
\label{tab:uno}
\end{table}

\begin{table}
\caption{Comparison of Scenarios ($\alpha=0.5$).}
\begin{center}
\begin{tabular}{cc|ccc|ccc|ccc}
\toprule
\multicolumn{2}{c|}{} & \multicolumn{3}{c|}{$\lambda=1$, $\sigma=1$} &
\multicolumn{3}{c|}{$\lambda=1.1$, $\sigma=1$} &
\multicolumn{3}{c}{$\lambda=1$, $\sigma=0.9$}\\
\midrule
$H$& $L$&$\widehat{MSE}$& $e(\widehat{Var})$ &$\widehat{Bias}^2$& $\widehat{MSE}$&
$e(\widehat{Var})$ &$\widehat{Bias}^2$ & $\widehat{MSE}$& $e(\widehat{Var})$
&$\widehat{Bias}^2$\tabularnewline
\midrule
\multirow{5}{*}{0.25}&3&$0.314$&$0.295$&$0.017$&$0.355$&$0.325$&$0.028$&$0.337$&$0.312$&$0.023$\tabularnewline
&4&$0.275$&$0.259$&$0.015$&$0.368$&$0.335$&$0.031$&$0.289$&$0.265$&$0.022$\tabularnewline
&5&$0.264$&$0.248$&$0.014$&$0.348$&$0.320$&$0.027$&$0.272$&$0.250$&$0.021$\tabularnewline
&6&$0.261$&$0.245$&$0.014$&$0.344$&$0.318$&$0.025$&$0.264$&$0.244$&$0.019$\tabularnewline
&7&$0.265$&$0.248$&$0.016$&$0.349$&$0.321$&$0.027$&$0.264$&$0.243$&$0.019$\tabularnewline
\midrule
\multirow{5}{*}{0.40}&3&$0.147$&$0.143$&$0.002$&$0.169$&$0.166$&$0.001$&$0.145$&$0.141$&$0.003$\tabularnewline
&4&$0.137$&$0.133$&$0.002$&$0.154$&$0.152$&$0.000$&$0.135$&$0.132$&$0.002$\tabularnewline
&5&$0.132$&$0.128$&$0.002$&$0.149$&$0.147$&$0.000$&$0.133$&$0.130$&$0.002$\tabularnewline
&6&$0.131$&$0.128$&$0.002$&$0.151$&$0.149$&$0.000$&$0.130$&$0.127$&$0.002$\tabularnewline
&7&$0.132$&$0.129$&$0.002$&$0.150$&$0.148$&$0.000$&$0.130$&$0.127$&$0.001$\tabularnewline
\midrule
\multirow{5}{*}{0.55}&3&$0.198$&$0.169$&$0.030$&$0.256$&$0.224$&$0.043$&$0.207$&$0.174$&$0.027$\tabularnewline
&4&$0.197$&$0.167$&$0.028$&$0.254$&$0.211$&$0.041$&$0.202$&$0.166$&$0.035$\tabularnewline
&5&$0.186$&$0.157$&$0.027$&$0.240$&$0.198$&$0.041$&$0.189$&$0.156$&$0.031$\tabularnewline
&6&$0.173$&$0.148$&$0.023$&$0.212$&$0.178$&$0.032$&$0.179$&$0.148$&$0.029$\tabularnewline
&7&$0.173$&$0.149$&$0.023$&$0.225$&$0.188$&$0.036$&$0.180$&$0.148$&$0.030$\tabularnewline
\midrule
\multirow{5}{*}{0.70}&3&$0.174$&$0.169$&$0.001$&$0.254$&$0.230$&$0.020$&$0.170$&$0.166$&$0.000$\tabularnewline
&4&$0.207$&$0.201$&$0.002$&$0.237$&$0.229$&$0.005$&$0.216$&$0.210$&$0.003$\tabularnewline
&5&$0.179$&$0.174$&$0.002$&$0.218$&$0.211$&$0.004$&$0.184$&$0.180$&$0.001$\tabularnewline
&6&$0.170$&$0.166$&$0.001$&$0.208$&$0.202$&$0.003$&$0.178$&$0.175$&$0.001$\tabularnewline
&7&$0.173$&$0.168$&$0.001$&$0.208$&$0.203$&$0.002$&$0.175$&$0.172$&$0.001$\tabularnewline
\midrule
\multirow{5}{*}{0.85}&3&$0.201$&$0.188$&$0.002$&$0.204$&$0.193$&$0.002$&$0.168$&$0.158$&$0.001$\tabularnewline
&4&$0.195$&$0.181$&$0.004$&$0.206$&$0.191$&$0.006$&$0.167$&$0.154$&$0.005$\tabularnewline
&5&$0.177$&$0.164$&$0.003$&$0.189$&$0.175$&$0.006$&$0.156$&$0.144$&$0.005$\tabularnewline
&6&$0.168$&$0.156$&$0.002$&$0.175$&$0.163$&$0.004$&$0.142$&$0.131$&$0.004$\tabularnewline
&7&$0.166$&$0.155$&$0.002$&$0.175$&$0.163$&$0.004$&$0.145$&$0.134$&$0.004$\tabularnewline
\bottomrule
\end{tabular}\end{center}
\label{tab:dos}
\end{table}

\begin{enumerate}
\item The empirical MSE ($\widehat{MSE}$) seems to decrease for a fixed scenario
  when the number of filters (number of moment conditions $L$) increases.
  For $\alpha=0.1$, we observe an
  increase in the empirical MSE as the Hurst parameter increases; this
  fact does not persist when $\alpha=0.5$, where the empirical MSEs are larger by up to a
  factor of 3, particularly when $H$ is small. The case of only 600 datapoints with a larger
  time interval appears to reach the limit of our joint estimation method's reliability when
  trying to
  estimate all three paramters together, since the MSE there exceeds 10\% of the value of
  the parameter routinely in this case. See however, the comments below when $\lambda$ is
  assumed known. 
\item As the drift and diffusion parameters $\lambda$ and $\sigma$ are perturbed,
  the empirical MSE seems to
  change mainly due to the change in the magnitude of the
  variance when that happens.
\item Sensitivity of $\widehat{Bias}^2$ and $e(\widehat{Var})$ to $\lambda$ and $\sigma$
  are noted as moderate, and for each case the change persists as the number of filters
  increases. Dependence of $e(\widehat{Var})$ on $\sigma$ is easily explained in theory
  since the asymptotic covariance in Proposition \ref{sec:fract-ornst-uhlenb}
  is proportional to $\sigma^2$. The two empirical statistics appear to be largely
  insensitive to the number of filters, however.
\item If we compare Tables \ref{tab:uno} and \ref{tab:dos} with Tables
  \ref{tab:lambdafixed1} and \ref{tab:lambdafixed2} in Appendix \ref{sec:additional-tables}, where the estimation is
  performed keeping the parameter $\lambda$ fixed on each scenario, we notice
  that most of the empirical MSE is due to the variation in $\lambda$
  estimation. The level of precision in essentially all scenarios is higher than 100 to 1,
  and the estimators can be considered as empirically unbiased. This holds 
  even in the unfavorable case of 600 observation times and larger time step
  (Table \ref{tab:lambdafixed2}). We also note that sensitivity of the estimators' 
  performances to the number of filters is very low, except for large values of $H$ where it
  seems preferably to choose at least $L=3$ even though there are only two parameters
  to estimate.
\item Finally Tables \ref{tab:wavelets1} and \ref{tab:wavelets2} in Appendix \ref{sec:additional-tables} contain the
  joint estimation of the three parameters under the same scenarios when the filters are
  Daubechies wavelets. We note that if we compare those tables with Tables
  \ref{tab:uno} and \ref{tab:dos}, 
  performance is essentially unchanged from one filter-class to the other. The slight
  improvement from Daubechies filters compared to finite-difference filters when the 
  Hurst parameter is larger than $\frac 1 2$ does not seem to be significant.    
\end{enumerate}

\bibliographystyle{amsplain}
\bibliography{biblioteca}

\appendix

\pagebreak

\section{Additional Tables}
\label{sec:additional-tables}

\begin{table}[!h]
\caption{Comparison of Scenarios ($\alpha=0.1$). Estimation of $(H,\sigma)$
  with $\lambda$ fixed.}
{\footnotesize
\begin{center}
\begin{tabular}{cc|ccc|ccc|ccc}
\toprule
\multicolumn{2}{c|}{} & \multicolumn{3}{c|}{$\lambda=1$, $\sigma=1$} &
\multicolumn{3}{c|}{$\lambda=1.1$, $\sigma=1$} &
\multicolumn{3}{c}{$\lambda=1$, $\sigma=0.9$}\\
\midrule
$H$& $L$&$\widehat{MSE}$& $e(\widehat{Var})$ &$\widehat{Bias}^2$& $\widehat{MSE}$&
$e(\widehat{Var})$ &$\widehat{Bias}^2$ & $\widehat{MSE}$& $e(\widehat{Var})$
&$\widehat{Bias}^2$\tabularnewline
\midrule
\multirow{5}{*}{0.25}&2&$0.00174$&$0.00166$&$1.26 \times 10^{-06}$&$0.00180$&$0.00171$&$1.58 \times 10^{-06}$&$0.00153$&$0.00145$&$1.30 \times 10^{-06}$\tabularnewline
&3&$0.00134$&$0.00125$&$5.79 \times 10^{-07}$&$0.00130$&$0.00121$&$3.81 \times 10^{-07}$&$0.00116$&$0.00108$&$4.23 \times 10^{-07}$\tabularnewline
&4&$0.00128$&$0.00119$&$2.62 \times 10^{-07}$&$0.00128$&$0.00119$&$2.87 \times 10^{-07}$&$0.00113$&$0.00104$&$5.19 \times 10^{-07}$\tabularnewline
&5&$0.00128$&$0.00120$&$2.92 \times 10^{-07}$&$0.00128$&$0.00119$&$2.05 \times 10^{-07}$&$0.00113$&$0.00105$&$5.36 \times 10^{-07}$\tabularnewline
&6&$0.00128$&$0.00120$&$1.11 \times 10^{-07}$&$0.00128$&$0.00119$&$1.57 \times 10^{-07}$&$0.00113$&$0.00104$&$5.29 \times 10^{-07}$\tabularnewline
\midrule
\multirow{5}{*}{0.40}&2&$0.00275$&$0.00268$&$1.15 \times 10^{-06}$&$0.00290$&$0.00281$&$2.32 \times 10^{-05}$&$0.00235$&$0.00227$&$6.58 \times 10^{-07}$\tabularnewline
&3&$0.00218$&$0.00211$&$3.49 \times 10^{-06}$&$0.00225$&$0.00217$&$1.15 \times 10^{-05}$&$0.00175$&$0.00167$&$1.54 \times 10^{-07}$\tabularnewline
&4&$0.00215$&$0.00207$&$4.15 \times 10^{-06}$&$0.00218$&$0.00210$&$9.43 \times 10^{-06}$&$0.00170$&$0.00163$&$1.72 \times 10^{-07}$\tabularnewline
&5&$0.00213$&$0.00206$&$4.25 \times 10^{-06}$&$0.00218$&$0.00210$&$9.19 \times 10^{-06}$&$0.00170$&$0.00163$&$2.50 \times 10^{-07}$\tabularnewline
&6&$0.00214$&$0.00206$&$4.49 \times 10^{-06}$&$0.00217$&$0.00209$&$8.77 \times 10^{-06}$&$0.00169$&$0.00161$&$2.73 \times 10^{-07}$\tabularnewline
\midrule
\multirow{5}{*}{0.55}&2&$0.00442$&$0.00435$&$1.62 \times 10^{-05}$&$0.00413$&$0.00404$&$1.65 \times 10^{-05}$&$0.00372$&$0.00364$&$1.38 \times 10^{-05}$\tabularnewline
&3&$0.00318$&$0.00311$&$1.76 \times 10^{-05}$&$0.00307$&$0.00299$&$1.54 \times 10^{-05}$&$0.00273$&$0.00267$&$1.44 \times 10^{-06}$\tabularnewline
&4&$0.00312$&$0.00304$&$1.71 \times 10^{-05}$&$0.00306$&$0.00298$&$1.87 \times 10^{-05}$&$0.00268$&$0.00262$&$1.12 \times 10^{-06}$\tabularnewline
&5&$0.00312$&$0.00305$&$1.60 \times 10^{-05}$&$0.00306$&$0.00298$&$1.84 \times 10^{-05}$&$0.00267$&$0.00261$&$1.05 \times 10^{-06}$\tabularnewline
&6&$0.00311$&$0.00304$&$1.74 \times 10^{-05}$&$0.00306$&$0.00298$&$1.85 \times 10^{-05}$&$0.00266$&$0.00260$&$9.60 \times 10^{-07}$\tabularnewline
\midrule
\multirow{5}{*}{0.70}&2&$0.00817$&$0.00793$&$1.90 \times 10^{-04}$&$0.00837$&$0.00827$&$4.23 \times 10^{-05}$&$0.00669$&$0.00652$&$1.16 \times 10^{-04}$\tabularnewline
&3&$0.00426$&$0.00421$&$5.33 \times 10^{-06}$&$0.00462$&$0.00458$&$2.40 \times 10^{-06}$&$0.00352$&$0.00347$&$4.90 \times 10^{-06}$\tabularnewline
&4&$0.00426$&$0.00422$&$5.53 \times 10^{-06}$&$0.00458$&$0.00454$&$2.75 \times 10^{-06}$&$0.00353$&$0.00349$&$4.87 \times 10^{-06}$\tabularnewline
&5&$0.00426$&$0.00422$&$5.92 \times 10^{-06}$&$0.00458$&$0.00453$&$3.18 \times 10^{-06}$&$0.00353$&$0.00349$&$4.92 \times 10^{-06}$\tabularnewline
&6&$0.00425$&$0.00421$&$6.12 \times 10^{-06}$&$0.00457$&$0.00452$&$3.10 \times 10^{-06}$&$0.00355$&$0.00351$&$4.79 \times 10^{-06}$\tabularnewline
\midrule
\multirow{5}{*}{0.85}&2&$0.03361$&$0.03300$&$5.44 \times 10^{-04}$&$0.03137$&$0.03053$&$7.81 \times 10^{-04}$&$0.03046$&$0.02985$&$5.11 \times 10^{-04}$\tabularnewline
&3&$0.01167$&$0.01165$&$4.39 \times 10^{-06}$&$0.01035$&$0.01032$&$1.72 \times 10^{-05}$&$0.00957$&$0.00954$&$1.74 \times 10^{-05}$\tabularnewline
&4&$0.01172$&$0.01170$&$4.01 \times 10^{-06}$&$0.01034$&$0.01031$&$1.74 \times 10^{-05}$&$0.00947$&$0.00943$&$1.95 \times 10^{-05}$\tabularnewline
&5&$0.01170$&$0.01169$&$3.99 \times 10^{-06}$&$0.01025$&$0.01022$&$1.68 \times 10^{-05}$&$0.00947$&$0.00943$&$2.26 \times 10^{-05}$\tabularnewline
&6&$0.01171$&$0.01169$&$3.98 \times 10^{-06}$&$0.01026$&$0.01023$&$1.71 \times 10^{-05}$&$0.00945$&$0.00942$&$1.91 \times 10^{-05}$\tabularnewline
\bottomrule
\end{tabular}\end{center}
}
\label{tab:lambdafixed1}
\end{table}

\begin{table}[!h]
\caption{Comparison of Scenarios ($\alpha=0.5$). Estimation of $(H,\sigma)$
  with $\lambda$ fixed.}
{\footnotesize
\begin{center}
\begin{tabular}{cc|ccc|ccc|ccc}
\toprule
\multicolumn{2}{c|}{} & \multicolumn{3}{c|}{$\lambda=1$, $\sigma=1$} &
\multicolumn{3}{c|}{$\lambda=1.1$, $\sigma=1$} &
\multicolumn{3}{c}{$\lambda=1$, $\sigma=0.9$}\\
\midrule
$H$& $L$&$\widehat{MSE}$& $e(\widehat{Var})$ &$\widehat{Bias}^2$& $\widehat{MSE}$&
$e(\widehat{Var})$ &$\widehat{Bias}^2$ & $\widehat{MSE}$& $e(\widehat{Var})$
&$\widehat{Bias}^2$\tabularnewline
\midrule
\multirow{5}{*}{0.25}&2&$0.00179$&$0.00121$&$8.34 \times 10^{-07}$&$0.00189$&$0.00131$&$2.47 \times 10^{-06}$&$0.00159$&$0.00106$&$5.11 \times 10^{-07}$\tabularnewline
&3&$0.00179$&$0.00122$&$1.12 \times 10^{-06}$&$0.00186$&$0.00129$&$3.50 \times 10^{-06}$&$0.00159$&$0.00106$&$6.27 \times 10^{-07}$\tabularnewline
&4&$0.00180$&$0.00122$&$1.18 \times 10^{-06}$&$0.00185$&$0.00128$&$3.79 \times 10^{-06}$&$0.00158$&$0.00106$&$6.82 \times 10^{-07}$\tabularnewline
&5&$0.00178$&$0.00121$&$1.41 \times 10^{-06}$&$0.00185$&$0.00128$&$3.72 \times 10^{-06}$&$0.00158$&$0.00105$&$7.32 \times 10^{-07}$\tabularnewline
&6&$0.00179$&$0.00121$&$1.26 \times 10^{-06}$&$0.00186$&$0.00128$&$3.53 \times 10^{-06}$&$0.00158$&$0.00106$&$1.05 \times 10^{-06}$\tabularnewline
\midrule
\multirow{5}{*}{0.40}&2&$0.00169$&$0.00124$&$7.24 \times 10^{-06}$&$0.00172$&$0.00125$&$2.14 \times 10^{-06}$&$0.00152$&$0.00111$&$5.87 \times 10^{-08}$\tabularnewline
&3&$0.00171$&$0.00125$&$7.11 \times 10^{-06}$&$0.00172$&$0.00125$&$1.99 \times 10^{-06}$&$0.00153$&$0.00112$&$6.10 \times 10^{-08}$\tabularnewline
&4&$0.00170$&$0.00124$&$7.28 \times 10^{-06}$&$0.00172$&$0.00125$&$1.90 \times 10^{-06}$&$0.00153$&$0.00112$&$7.78 \times 10^{-08}$\tabularnewline
&5&$0.00170$&$0.00124$&$7.72 \times 10^{-06}$&$0.00171$&$0.00125$&$1.77 \times 10^{-06}$&$0.00152$&$0.00111$&$1.34 \times 10^{-07}$\tabularnewline
&6&$0.00170$&$0.00124$&$7.88 \times 10^{-06}$&$0.00170$&$0.00124$&$1.70 \times 10^{-06}$&$0.00152$&$0.00111$&$1.77 \times 10^{-07}$\tabularnewline
\midrule
\multirow{5}{*}{0.55}&2&$0.00183$&$0.00150$&$1.02 \times 10^{-05}$&$0.00179$&$0.00149$&$3.18 \times 10^{-07}$&$0.00152$&$0.00123$&$1.95 \times 10^{-07}$\tabularnewline
&3&$0.00180$&$0.00149$&$1.07 \times 10^{-05}$&$0.00175$&$0.00145$&$3.71 \times 10^{-07}$&$0.00147$&$0.00119$&$1.60 \times 10^{-07}$\tabularnewline
&4&$0.00180$&$0.00149$&$1.05 \times 10^{-05}$&$0.00175$&$0.00145$&$3.68 \times 10^{-07}$&$0.00146$&$0.00118$&$1.54 \times 10^{-07}$\tabularnewline
&5&$0.00179$&$0.00148$&$9.46 \times 10^{-06}$&$0.00173$&$0.00143$&$4.20 \times 10^{-07}$&$0.00146$&$0.00118$&$2.30 \times 10^{-07}$\tabularnewline
&6&$0.00179$&$0.00148$&$9.85 \times 10^{-06}$&$0.00175$&$0.00145$&$3.33 \times 10^{-07}$&$0.00146$&$0.00118$&$2.49 \times 10^{-07}$\tabularnewline
\midrule
\multirow{5}{*}{0.70}&2&$0.00305$&$0.00279$&$2.82 \times 10^{-05}$&$0.00334$&$0.00313$&$6.76 \times 10^{-06}$&$0.00288$&$0.00265$&$2.06 \times 10^{-05}$\tabularnewline
&3&$0.00270$&$0.00248$&$7.47 \times 10^{-06}$&$0.00308$&$0.00288$&$4.75 \times 10^{-06}$&$0.00250$&$0.00231$&$3.99 \times 10^{-06}$\tabularnewline
&4&$0.00270$&$0.00248$&$6.60 \times 10^{-06}$&$0.00308$&$0.00288$&$4.76 \times 10^{-06}$&$0.00251$&$0.00231$&$4.01 \times 10^{-06}$\tabularnewline
&5&$0.00269$&$0.00247$&$6.13 \times 10^{-06}$&$0.00308$&$0.00288$&$4.64 \times 10^{-06}$&$0.00249$&$0.00229$&$3.51 \times 10^{-06}$\tabularnewline
&6&$0.00269$&$0.00247$&$6.17 \times 10^{-06}$&$0.00306$&$0.00286$&$4.71 \times 10^{-06}$&$0.00247$&$0.00227$&$3.06 \times 10^{-06}$\tabularnewline
\midrule
\multirow{5}{*}{0.85}&2&$0.01554$&$0.01513$&$3.02 \times 10^{-04}$&$0.01501$&$0.01457$&$3.46 \times 10^{-04}$&$0.01298$&$0.01233$&$5.28 \times 10^{-04}$\tabularnewline
&3&$0.01008$&$0.00999$&$2.00 \times 10^{-05}$&$0.01027$&$0.01018$&$1.46 \times 10^{-05}$&$0.00838$&$0.00826$&$4.72 \times 10^{-05}$\tabularnewline
&4&$0.00998$&$0.00989$&$1.51 \times 10^{-05}$&$0.01016$&$0.01007$&$1.43 \times 10^{-05}$&$0.00830$&$0.00818$&$5.12 \times 10^{-05}$\tabularnewline
&5&$0.00974$&$0.00965$&$1.41 \times 10^{-05}$&$0.00991$&$0.00983$&$1.25 \times 10^{-05}$&$0.00806$&$0.00794$&$5.39 \times 10^{-05}$\tabularnewline
&6&$0.00974$&$0.00965$&$1.37 \times 10^{-05}$&$0.00962$&$0.00954$&$1.24 \times 10^{-05}$&$0.00798$&$0.00786$&$4.55 \times 10^{-05}$\tabularnewline
\bottomrule
\end{tabular}\end{center}
}
\label{tab:lambdafixed2}
\end{table}

\begin{table}[!h]
\caption{Comparison of Scenarios ($\alpha=0.1$). Estimation of $(H,\lambda,\sigma)$
  with Daubechies wavelets.}
\begin{center}
\begin{tabular}{cc|ccc|ccc|ccc}
\toprule
\multicolumn{2}{c|}{} & \multicolumn{3}{c|}{$\lambda=1$, $\sigma=1$} &
\multicolumn{3}{c|}{$\lambda=1.1$, $\sigma=1$} &
\multicolumn{3}{c}{$\lambda=1$, $\sigma=0.9$}\\
\midrule
$H$& $L$&$\widehat{MSE}$& $e(\widehat{Var})$ &$\widehat{Bias}^2$& $\widehat{MSE}$&
$e(\widehat{Var})$ &$\widehat{Bias}^2$ & $\widehat{MSE}$& $e(\widehat{Var})$
&$\widehat{Bias}^2$\tabularnewline
\midrule
\multirow{5}{*}{0.25}&3&$0.045$&$0.044$&$0.000$&$0.053$&$0.052$&$0.000$&$0.045$&$0.044$&$0.000$\tabularnewline
&4&$0.042$&$0.041$&$0.000$&$0.047$&$0.046$&$0.000$&$0.042$&$0.040$&$0.000$\tabularnewline
&5&$0.042$&$0.040$&$0.000$&$0.047$&$0.046$&$0.000$&$0.041$&$0.040$&$0.000$\tabularnewline
&6&$0.041$&$0.040$&$0.000$&$0.047$&$0.045$&$0.000$&$0.040$&$0.039$&$0.000$\tabularnewline
&7&$0.041$&$0.040$&$0.000$&$0.047$&$0.045$&$0.000$&$0.040$&$0.039$&$0.000$\tabularnewline
\midrule
\multirow{5}{*}{0.40}&3&$0.061$&$0.059$&$0.000$&$0.077$&$0.074$&$0.001$&$0.064$&$0.061$&$0.001$\tabularnewline
&4&$0.059$&$0.056$&$0.000$&$0.074$&$0.071$&$0.001$&$0.063$&$0.060$&$0.001$\tabularnewline
&5&$0.060$&$0.057$&$0.001$&$0.074$&$0.071$&$0.001$&$0.062$&$0.059$&$0.001$\tabularnewline
&6&$0.059$&$0.057$&$0.001$&$0.074$&$0.071$&$0.001$&$0.063$&$0.059$&$0.001$\tabularnewline
&7&$0.060$&$0.058$&$0.001$&$0.075$&$0.072$&$0.001$&$0.063$&$0.060$&$0.001$\tabularnewline
\midrule
\multirow{5}{*}{0.55}&3&$0.071$&$0.064$&$0.004$&$0.087$&$0.081$&$0.003$&$0.068$&$0.063$&$0.003$\tabularnewline
&4&$0.072$&$0.065$&$0.004$&$0.084$&$0.079$&$0.002$&$0.069$&$0.064$&$0.003$\tabularnewline
&5&$0.071$&$0.064$&$0.004$&$0.086$&$0.081$&$0.002$&$0.069$&$0.063$&$0.003$\tabularnewline
&6&$0.070$&$0.064$&$0.003$&$0.086$&$0.081$&$0.002$&$0.069$&$0.063$&$0.003$\tabularnewline
&7&$0.070$&$0.064$&$0.004$&$0.087$&$0.081$&$0.002$&$0.068$&$0.063$&$0.003$\tabularnewline
\midrule
\multirow{5}{*}{0.70}&3&$0.136$&$0.125$&$0.007$&$0.174$&$0.157$&$0.012$&$0.137$&$0.123$&$0.010$\tabularnewline
&4&$0.116$&$0.107$&$0.004$&$0.137$&$0.127$&$0.005$&$0.110$&$0.101$&$0.005$\tabularnewline
&5&$0.105$&$0.098$&$0.002$&$0.126$&$0.118$&$0.004$&$0.102$&$0.095$&$0.004$\tabularnewline
&6&$0.101$&$0.094$&$0.002$&$0.123$&$0.115$&$0.003$&$0.097$&$0.091$&$0.003$\tabularnewline
&7&$0.102$&$0.095$&$0.002$&$0.121$&$0.113$&$0.003$&$0.099$&$0.092$&$0.003$\tabularnewline
\midrule
\multirow{5}{*}{0.85}&3&$0.268$&$0.240$&$0.003$&$0.284$&$0.247$&$0.001$&$0.235$&$0.202$&$0.007$\tabularnewline
&4&$0.184$&$0.167$&$0.006$&$0.208$&$0.190$&$0.007$&$0.172$&$0.156$&$0.006$\tabularnewline
&5&$0.168$&$0.152$&$0.005$&$0.189$&$0.173$&$0.006$&$0.151$&$0.138$&$0.004$\tabularnewline
&6&$0.157$&$0.142$&$0.004$&$0.175$&$0.160$&$0.005$&$0.142$&$0.128$&$0.005$\tabularnewline
&7&$0.154$&$0.139$&$0.004$&$0.173$&$0.157$&$0.005$&$0.138$&$0.125$&$0.004$\tabularnewline
\bottomrule
\end{tabular}
\label{tab:wavelets1}
\end{center}
\end{table}

\begin{table}[!tbp]
\caption{Comparison of Scenarios ($\alpha=0.5$). Estimation of $(H,\lambda,\sigma)$
  with Daubechies wavelets.}
\begin{center}
\begin{tabular}{cc|ccc|ccc|ccc}
\toprule
\multicolumn{2}{c|}{} & \multicolumn{3}{c|}{$\lambda=1$, $\sigma=1$} &
\multicolumn{3}{c|}{$\lambda=1.1$, $\sigma=1$} &
\multicolumn{3}{c}{$\lambda=1$, $\sigma=0.9$}\\
\midrule
$H$& $L$&$\widehat{MSE}$& $e(\widehat{Var})$ &$\widehat{Bias}^2$& $\widehat{MSE}$&
$e(\widehat{Var})$ &$\widehat{Bias}^2$ & $\widehat{MSE}$& $e(\widehat{Var})$
&$\widehat{Bias}^2$\tabularnewline
\midrule
\multirow{5}{*}{0.25}&3&$0.217$&$0.203$&$0.012$&$0.216$&$0.213$&$0.002$&$0.212$&$0.202$&$0.009$\tabularnewline
&4&$0.305$&$0.280$&$0.023$&$0.344$&$0.332$&$0.010$&$0.338$&$0.315$&$0.022$\tabularnewline
&5&$0.306$&$0.283$&$0.022$&$0.349$&$0.335$&$0.011$&$0.322$&$0.301$&$0.019$\tabularnewline
&6&$0.299$&$0.276$&$0.021$&$0.360$&$0.345$&$0.013$&$0.318$&$0.297$&$0.020$\tabularnewline
&7&$0.306$&$0.280$&$0.024$&$0.358$&$0.343$&$0.012$&$0.319$&$0.297$&$0.021$\tabularnewline
\midrule
\multirow{5}{*}{0.40}&3&$0.126$&$0.124$&$0.000$&$0.129$&$0.125$&$0.002$&$0.120$&$0.118$&$0.000$\tabularnewline
&4&$0.134$&$0.133$&$0.000$&$0.160$&$0.158$&$0.000$&$0.133$&$0.131$&$0.000$\tabularnewline
&5&$0.137$&$0.135$&$0.000$&$0.156$&$0.154$&$0.000$&$0.130$&$0.128$&$0.000$\tabularnewline
&6&$0.134$&$0.132$&$0.000$&$0.159$&$0.157$&$0.000$&$0.128$&$0.126$&$0.000$\tabularnewline
&7&$0.136$&$0.134$&$0.000$&$0.157$&$0.155$&$0.000$&$0.126$&$0.125$&$0.000$\tabularnewline
\midrule
\multirow{5}{*}{0.55}&3&$0.197$&$0.169$&$0.027$&$0.249$&$0.212$&$0.047$&$0.191$&$0.169$&$0.028$\tabularnewline
&4&$0.195$&$0.167$&$0.026$&$0.246$&$0.201$&$0.043$&$0.189$&$0.162$&$0.026$\tabularnewline
&5&$0.194$&$0.166$&$0.026$&$0.235$&$0.192$&$0.041$&$0.178$&$0.152$&$0.024$\tabularnewline
&6&$0.194$&$0.165$&$0.028$&$0.224$&$0.184$&$0.038$&$0.174$&$0.148$&$0.025$\tabularnewline
&7&$0.194$&$0.165$&$0.027$&$0.221$&$0.184$&$0.035$&$0.167$&$0.145$&$0.020$\tabularnewline
\midrule
\multirow{5}{*}{0.70}&3&$0.202$&$0.167$&$0.031$&$0.232$&$0.194$&$0.035$&$0.192$&$0.161$&$0.029$\tabularnewline
&4&$0.187$&$0.182$&$0.002$&$0.228$&$0.222$&$0.002$&$0.179$&$0.175$&$0.001$\tabularnewline
&5&$0.181$&$0.177$&$0.001$&$0.216$&$0.211$&$0.001$&$0.169$&$0.166$&$0.001$\tabularnewline
&6&$0.176$&$0.172$&$0.001$&$0.212$&$0.208$&$0.001$&$0.166$&$0.163$&$0.001$\tabularnewline
&7&$0.176$&$0.172$&$0.001$&$0.214$&$0.210$&$0.001$&$0.165$&$0.161$&$0.001$\tabularnewline
\midrule
\multirow{5}{*}{0.85}&3&$0.202$&$0.188$&$0.004$&$0.212$&$0.195$&$0.007$&$0.175$&$0.161$&$0.006$\tabularnewline
&4&$0.198$&$0.183$&$0.004$&$0.204$&$0.187$&$0.007$&$0.172$&$0.157$&$0.007$\tabularnewline
&5&$0.192$&$0.177$&$0.005$&$0.195$&$0.179$&$0.006$&$0.165$&$0.151$&$0.007$\tabularnewline
&6&$0.194$&$0.180$&$0.004$&$0.195$&$0.179$&$0.006$&$0.165$&$0.150$&$0.007$\tabularnewline
&7&$0.189$&$0.175$&$0.005$&$0.193$&$0.178$&$0.005$&$0.167$&$0.152$&$0.007$\tabularnewline
\bottomrule
\end{tabular}\end{center}
\label{tab:wavelets2}
\end{table}

\FloatBarrier
\section{Proofs.}
\label{sec:appendix-proofs}

\subsection*{Proof of Lemma \ref{lemaunif1}}
  First recall that $E[\varphi(t)^2]=V(\theta_0)$, where stationarity allows the above expression does not depend on $t$, and the data's true parameter is $\theta_0$. Note that, for any $\theta \in \Theta$ and using the ergodicity of $X_t$:
  \begin{align*}
    \|\hat {\nn g}_N(\theta)-\nn g_0(\theta)\|&=\left \|\frac{1}{N-L+1}\sum_{i=L}^N [\nns{\varphi}(t_i)^2-\nn{V}(\theta)]-[\mathbf{V}(\theta_0)-\mathbf{V}(\theta)]\right\|\\
&=\left \|\frac{1}{N-L+1}\sum_{i=L}^N \nns{\varphi}(t_i)^2-\mathbf{V}(\theta_0)\right\| \stackrel{a.s}{\rightarrow}0
  \end{align*}
where $\nns{\varphi}(t)^2:=(\varphi_1(t)^2,\ldots,\varphi_L(t)^2)'$ and $\mathbf{V}(\theta):=(V_1(\theta),\ldots,V_L(\theta))'$. Then we have:
\begin{align}\label{unifghat}
  \sup_{\theta \in \Theta}\|\hat {\nn g}_N(\theta)-\nn g_0(\theta)\|\stackrel{a.s.}{\rightarrow} 0.
\end{align}
Furthermore, based on the proof of Theorem 2.6 in \cite{mcfadden}:
\begin{align*}
  &|\hat Q_N(\theta)-Q_0(\theta)|\\
  &\leq |(\hat {\nn g}_N(\theta)-\nn g_0(\theta))'A(\hat {\nn g}_N(\theta)-\nn g_0(\theta))|+2|(\hat {\nn g}_N(\theta)-\nn g_0(\theta))'A\nn g_0(\theta)|\\
  &\leq \|\hat {\nn g}_N(\theta)-\nn g_0(\theta)\|\cdot \|A\|_2+2 \|\hat {\nn g}_N(\theta)-\nn g_0(\theta)\|\cdot \|A\|\cdot \|g_0(\theta)\|\\
  &\stackrel{a.s.}{\rightarrow}0
\end{align*}
for any $\theta \in \Theta$. Here we used the following facts: (1) the matrix $A$ are deterministic, (2) $\nn g_0(\theta)$ is uniformly bounded (because $\rho_\theta(t)$ is continuous over the compact set $\Theta$) and (3) equation \eqref{unifghat}. Finally if we take the supremum over all $\theta \in \Theta$, the lemma holds.

\qed

\subsection*{Proof of Lemma \ref{unique}}
Note that $Q_0(\theta)$ can be written as (for any $\theta \in \Theta$):
\begin{align*}
  Q_0(\theta)&=\nn g_0(\theta)'A\nn g_0(\theta)\\
  &=[\mathbf{V}(\theta_0)-\mathbf{V}(\theta)]'A[\mathbf{V}(\theta_0)-\mathbf{V}(\theta)]\\
  &=
  \begin{bmatrix}
    \mathbf{b}_1'(\nns \rho_{\theta_0}(\alpha k)-\nns \rho_\theta(\alpha k))\\
    \vdots\\
    \mathbf{b}_L'(\nns \rho_{\theta_0}(\alpha k)-\nns \rho_\theta(\alpha k))
  \end{bmatrix}'A
  \begin{bmatrix}
    \mathbf{b}_1'(\nns \rho_{\theta_0}(\alpha k)-\nns \rho_\theta(\alpha k))\\
    \vdots\\
    \mathbf{b}_L'(\nns \rho_{\theta_0}(\alpha k)-\nns \rho_\theta(\alpha k))
  \end{bmatrix}\\
&=(\nns{\rho}_{\theta_0}(\alpha)-\nns{\rho}_\theta(\alpha))'B'AB(\nns{\rho}_{\theta_0}(\alpha)-\nns{\rho}_\theta(\alpha))
\end{align*}
Since $A>0$ and $B$ has column-rank between $p$ and $L$ (see assumption
\ref{sec:ap2}) we can diagonalize $B'AB$ to get:  
\begin{align*}
  Q_0(\theta)=\sum_{i=1}^p \lambda_i[\rho_{\theta_0}(\alpha\cdot
  (i-1))-\rho_\theta(\alpha\cdot (i-1))]^2+J(\theta)
\end{align*}
where $\{\lambda_i\}_{i=1}^p$ are non-zero
eigenvalues of $B'AB$ and $J(\theta)$ is a non-injective function of $\theta$
such that $J(\theta_0)=0$. Because of Assumption \ref{sec:ap1} the lemma holds.
\qed

\subsection*{Proof of Lemma \ref{asymnorm1}}
  \begin{itemize}
    \item[(i)]:
      Let $1\leq j\leq p$ and $\theta \in \Theta$. Then the $j$-th column of $\hat G_N(\theta)$ is:
      \begin{align*}
        (\hat
        G_N(\theta))_j&=\frac{1}{N-L+1}\sum_{i=L}^N\frac{\partial}{\partial
          \theta_j}\left(\nns \varphi(t_i)^2-\mathbf{V}(\theta)\right)\\
&=\frac{1}{N-L+1}\sum_{i=L}^N-\frac{\partial}{\partial \theta_j}\mathbf{V}(\theta)\\
&=\frac{1}{N-L+1}\sum_{i=L}^N\frac{\partial}{\partial \theta_j}\left(\mathbf{V}(\theta_0)-\mathbf{V}(\theta)\right)=-\frac{\partial}{\partial \theta_j}\mathbf{V}(\theta)\\
&=E\left[\frac{\partial}{\partial \theta_j} \hat{\mathbf{g}}_N(\theta)\right]\\
&=(G(\theta))_j.
      \end{align*}
      and the continuity is inherited from the continuous differentiability of $\rho_\theta(\cdot)$.
    \item[(ii)]: Using (i), we have:
      \begin{align*}
        G(\theta)=-B\nabla_\theta \nns{\rho}_\theta(\alpha)
      \end{align*}
      and this matrix is a full rank-matrix due to the remark after Assumption \ref{sec:ap1}
      and the fact that $B$ is full-rank too. Since $D$ is positive definite, the result holds.
  \end{itemize}

\qed

\subsection*{Proof of Lemma \ref{sec:asymptotic-normality}}
First note that by rescaling the vector $\hat {\nn g}_N(\theta_0)$ with the variance of each component of the vector $\nn g(t,\theta_0)$, we can write the sample moment equations in the following way:
\begin{align}\label{hermite2}
  \sqrt{N}V_D(\theta_0)^{-1}\hat {\nn g}_N(\theta_0)&=\frac{\sqrt{N}}{N-L+1}\sum_{i=L}^N\left(\frac{\varphi_\ell(t_i)^2-V_\ell(\theta_0)}{V_\ell(\theta_0)}\right)_{\ell\in \{1,\ldots,L\}}\notag \\ 
  &=\frac{\sqrt{N}}{N-L+1}\sum_{i=L}^N\left(Z_{\ell,t_i}^2-1\right)_{\ell\in \{1,\ldots,L\}}\notag \\ 
  &=\frac{\sqrt{N}}{N-L+1}\sum_{i=L}^N\left(H_2(Z_{\ell,t_i})\right)_{\ell\in \{1,\ldots,L\}}
\end{align}
where $V_D(\theta_0):=\text{Diag}\left(V_\ell(\theta_0)\right)_{\ell\in \{1,\ldots,L\}}$, $H_2(x)=x^2-1$ is the second Hermite polynomial (see \cite{NIST18}) and 
\begin{align*}
  Z_{\ell,t_i}:=\frac{\varphi_\ell(t_i)}{\sqrt{V_\ell(\theta_0)}}.
\end{align*}
So the main idea of the proof is to use the vector-valued version of the
Breuer-Major theorem with spectral-information conditions (see Theorem 3.1,
\cite{Bierme2011}) to study the asymptotic behavior of \eqref{hermite2}. In
order to achieve this, we need to analyze the spectral behavior of
$Z_{\ell,t_i}$. First recall that for $n,m \in \{0,\ldots,N\}$ we have (see
equation \eqref{eq:10}):
\begin{align*}
  \text{Cov}(X_{t_m},X_{t_n})=\frac{1}{2\pi}\int_{-\pi}^\pi e^{i|m-n|u}\bar{f}_{\theta_0}\left(\frac u \alpha\right)du.
\end{align*}
where $\bar{f}_{\theta_0}(\cdot)$ is the spectral density of the discretized
version of $X_t$ at $t \in \{t_0,\ldots,t_N\}$ (see equation
(\ref{eq:7})). Now, choose $\ell,\ell' \in \{1,\ldots,L\}$ and note that the
covariance among $Z_{\ell,t_m}$ and $Z_{\ell',t_n}$ is:  
\footnote{Without loss of generality assume that $n,m>L$.} 
\begin{align*}
\beta(m-n| \ell,\ell'):&=\text{Cov}(Z_{\ell,t_m},Z_{\ell',t_n})\\
&=\frac{1}{\sqrt{V_\ell(\theta_0)V_{\ell'}(\theta_0)}}\sum_{k,k'=0}^La_k^\ell
a_{k'}^{\ell'} \text{Cov}(X_{t_m-\alpha k},X_{t_n-\alpha k'})\\
  &=\frac{1}{\sqrt{V_\ell(\theta_0)V_{\ell'}(\theta_0)}} \sum_{k,k'=0}^L a_k^\ell
a_{k'}^{\ell'} \cdot \frac{1}{2\pi}\int_{-\pi}^\pi e^{i[(m-n)-(k-k')]u}\bar{f}_{\theta_0}(u/\alpha)du\\
  &=\int_{-\pi}^\pi e^{i(m-n)u}|P_{\nn a_\ell}(e^{-iu})|\cdot
  \overline{|P_{\nn a_{\ell'}}(e^{-iu})|}\frac{\bar{f}_{\theta_0}(u/\alpha)}{\sqrt{V_\ell(\theta_0)V_{\ell'}(\theta_0)}}du\\
  &=\int_{-\pi}^\pi e^{i(m-n)u}\bar h_{\theta_0,\alpha}(u|\ell,\ell')du,
\end{align*}
where $P_{\nn a_\ell}(x):=\sum_{k=0}^La_k^\ell x^k$ and $\bar
h_{\theta_0,\alpha}(u|\ell,\ell'):=|P_{a_\ell}(e^{-iu})|\cdot \overline{|P_{a_{\ell'}}(e^{-iu})|}\frac{\bar{f}_{\theta_0}(u/\alpha)}{\sqrt{V_\ell(\theta_0)V_{\ell'}(\theta_0)}}$. 

If we denote $\bar h_{\theta_0,\alpha}(\cdot):=\left(\bar
  h_{\theta_0,\alpha}(u|\ell,\ell')\right)_{\ell,\ell' \in \{1,\ldots,L\}}$
the symmetric matrix of spectral densities among different filters, it
is easy to conclude that this matrix does not depend on the sample size $N$,
because $\alpha$ is fixed. Therefore, in order to apply Theorem 3.1 of
\cite{Bierme2011}, we only need to check that $\bar h_{\theta_0,\alpha}(\cdot)
\in L^2((-\pi,\pi))$ uniformly. But this is true due to the fact that if
$l_\ell$ and $l_{\ell'}$ are the corresponding filter orders for $\ell$ and
$\ell'$ respectively, the  polynomial $P_{\nn
  a_\ell}(x)$ has $l_\ell$ vanishing moments and:
\begin{align*}
  \|\bar
  h_{\theta_0,\alpha}(u|\ell,\ell')\|^2_{L^2((-\pi,\pi))}&=\int_{-\pi}^\pi
  |P_{\nn a_\ell}(e^{-iu})|^2\cdot \overline{|P_{\nn a_{\ell'}}(e^{-iu})|}^2\frac{\bar{f}_{\theta_0}(u/\alpha)^2}{V_\ell(\theta_0)V_{\ell'}(\theta_0)}du\\
&\leq \frac{1}{V_\ell(\theta_0)V_{\ell'}(\theta_0) }\int_{-\pi}^\pi \min\left(1,|u|^{2l_{\ell}+2{l_{\ell'}}}\right)\bar{f}_{\theta_0}(u/\alpha)^2du<\infty
\end{align*}
by Assumption \ref{sec:ap4}. Then:
\begin{align*}
  \sqrt{N}V_D(\theta_0)^{-1}\hat {\nn g}_N(\theta_0)\stackrel{d}{\rightarrow}N(0,U)
\end{align*}
where, by Plancherel theorem: $U_{ij}=2\sum_{k \in \mathbb Z}\beta(k| i,j)^2=2\|\bar h_{\theta_0,\alpha}(u|i,j)\|^2_{L^2((-\pi,\pi))}$ and finally the lemma holds with:
\begin{align*}
  \Omega_{ij}&=V_i(\theta_0)U_{ij}V_j(\theta_0)=2\sum_{k\in \mathbb Z}E[\varphi_i(t_{n+k})\varphi_j(t_n)]^2\notag \\
  &=2\sum_{k\in \mathbb Z}\left[\sum_{q,q'=0}^La_q^ia_{q'}^jE\left[X_{\alpha(n+k)-\alpha q}X_{\alpha n-\alpha q'}\right]\right]^2\\
  &=2\sum_{k\in \mathbb Z}\left[\sum_{p=0}^Lb_p^{i,j}\rho_{\theta_0}\left[\alpha(k+p)\right]\right]^2
\end{align*}
for $i,j=1\ldots,L$.

\qed

\subsection*{Proof of Lemma \ref{sec:LemacotaEPsi}}
  For any matrix norm $\|\cdot \|$ and $N\in \mathbb N$, we have that:
  \begin{align*}
    \|\psi_N\|=\|(G'(\theta_N)A G'(\theta_N))^{-1}\|\cdot \|G'\|_{op}\cdot \|A\|,
  \end{align*}
where $\|\cdot \|_{op}$ is the operator norm. Because of the remark after Assumption \ref{sec:ap1} and Lemma \ref{asymnorm1}, the operator norm in the above equation is bounded for all $N$. Also note that if $F_N:=G'(\theta_N)A G(\theta_N)$ and if $\lambda_{\text{max}}(A)$ and $\lambda_{\text{min}}(A)$ are the maximum and minimum eigenvalue of $A$:
\begin{align*}
  \|F_N^{-1}\|\leq p \cdot \lambda_{\text{min}}(F_N)^{-1}&=p \cdot \left[\inf_{\|x\|=1}x'G'(\theta_N)A G(\theta_N)x\right]^{-1} \\
&\leq p \left[\inf_{\|u\|\geq 1}u'A u\right]^{-1}=p \cdot \lambda_{\text{min}}(A)^{-1},
\end{align*}
where $u=G(\theta_N)x$. Moreover, 
\begin{align*}
  \|A\|\leq (L+1)\cdot \lambda_{\text{max}}(A)
\end{align*}
and hence there exists $R_{s,L}>0$ such that $\|\psi_N\|^{4s}<R_{s,L}$ uniformly for all $N$. The lemma holds.

\qed

\subsection*{Proof of Lemma \ref{sec:asymptotic-normality-2}}
    Take $l,l' \in \{l_0,\ldots,l_L\}$, hence:
    \begin{align}\label{ROUcase}
      \int_{-\pi}^\pi R(u| l,l')^2du&=\int_{-\pi}^\pi\min\left(1,|u|^{2(l+l')}\right)\left[c_H\sum_{p\in \mathbb Z}\frac{\left|\frac{u+2\pi p}{\alpha}\right|^{1-2H}}{\left(\frac{u+2\pi p}{\alpha}\right)^2+\lambda^2}\right]^2du\notag \\
      &=\alpha^{4H+2}c_H^2\int_{1\leq |u|\leq \pi}\left(\sum_{p\in \mathbb Z}\frac{|u+2\pi p|^{1-2H}}{(u+2\pi p)^2+(\lambda \alpha)^2}\right)^2du+\notag \\& \quad \alpha^{4H+2}c_H^2\int_{|u|\leq 1}|u|^{2(l+l')}\left(\sum_{p\in \mathbb Z}\frac{|u+2\pi p|^{1-2H}}{(u+2\pi p)^2+(\lambda \alpha)^2}\right)^2du.
  \end{align}
  First note that for any $0<H<1$ and $p\neq 0$:
  \begin{align*}
    \sum_{p\neq 0, p\in \mathbb Z}\frac{|u+2\pi p|^{1-2H}}{(u+2\pi p)^2+(\lambda \alpha)^2}\leq 2\sum_{p=1}^\infty(u+2\pi p)^{-1-2H}<\infty
  \end{align*}
  and this implies that the two summands in equation (\ref{ROUcase}) are
  finite for all $p$ if $\pi<|u|<1$. The same applies if $|u|<1$ and $p\neq
  0$. If $p=0$ and $|u|<1$ then the situation is slightly more
  complicated. Note that:
    \begin{align*}
      \int_{|u|<1}|u|^{2(l'+l)}\frac{|u|^{2-4H}}{[u^2+(\lambda \alpha)^2]^2}du\leq \frac{2}{(\lambda \alpha)^4}\int_0^1u^{2(l'+l)+2-4H}du<\infty
    \end{align*}
if and only if $2(l+l')>4H-3$. Then the lemma holds.

\qed

\begin{lema}\label{lemacovfOU}
Assume that $X_t$ is the stationary solution of the Ornstein-Uhlenbeck
SDE. Then, for $t\geq 0$, the autocovariance function $\rho_\theta(t)$ of
$\theta=(H,\lambda,\sigma)$ is\footnote{Note that $F_H(\cdot)$ is the cdf of a $\Gamma(2H-1,1)$ random variable when $H\geq 1/2$.}:
\begin{align}\label{covmedio}
\rho_\theta(t)=e^{-\lambda t}\text{Var}(X_0)\left[\frac{1+e^{2\lambda
      t}}{2}-\lambda B_{\theta}(t)\right]
\end{align}
where:
\begin{align*}
\text{Var}(X_0)&=\sigma^2\lambda^{-2H}H\Gamma(2H),\\
B_{\theta}(t)&=\int_0^t e^{2\lambda v}F_H(\lambda v)dv\\
F_H(x)&=\frac{1}{\Gamma(2H-1)}\int_0^xe^{-s}s^{2H-2}ds.
\end{align*}
\end{lema}
\begin{proof}
Note that:
\begin{align*}
\rho_\theta(t)=E[X_tX_0]&=\sigma^2E\left[\int_{-\infty}^0 e^{\lambda u}dB_u^H
  \int_{-\infty}^t e^{-\lambda(t-v)}dB_v^H\right]\\
&=e^{-\lambda t}\left[\text{Var}(X_0)+\sigma^2E\left[\int_{-\infty}^0 e^{\lambda u}dB_u^H
  \int_0^t e^{-\lambda v}dB_v^H\right]\right].
\end{align*}
Using the proof of Lemma 5.2 in \cite{Hu} and Lemma 2.1 of \cite{cheridito}: 
\begin{align}\label{eq:3}
  \text{Var}(X_0)&=\sigma^2\lambda^{-2H}H\Gamma(2H)
\end{align}
and
\begin{align*}
  &\sigma^2E\left[\int_{-\infty}^0 e^{\lambda u}dB_u^H\int_0^t e^{-\lambda
      v}dB_v^H\right]=\sigma^2 H(2H-1)\int_{-\infty}^0\int_0^te^{\lambda
    (u+v)}|u-v|^{2H-2}dvdu\\
  &=\sigma^2 H(2H-1)\int_0^t\int_v^\infty e^{2\lambda v-\lambda
    x}x^{2H-2}dxdv\\
&=\sigma^2 H\Gamma(2H)\lambda^{1-2H}\int_0^te^{2\lambda
  v}dv-\sigma^2H(2H-1)\int_0^te^{2\lambda v}\int_0^ve^{-\lambda x}x^{2H-2}dx
dv\\
&=\text{Var}(X_0)\left[\frac{e^{2\lambda t}-1}{2}-\frac{\lambda}{\Gamma(2H-1)} \int_o^te^{2\lambda
  v}\int_0^{\lambda v}e^{-s}s^{2H-2}ds dv\right]
\end{align*}
and hence:
\begin{align*}
\rho_\theta(t)=e^{-\lambda t} \text{Var}(X_0)\left[1+\frac{e^{2\lambda
      t}-1}{2}-\lambda B_\theta(t)\right]
\end{align*}
and \eqref{covmedio} holds. Note that \cite{Hu} proves formula \eqref{eq:3} in
the case $H\geq \frac 1 2$, but using the analytical continuation of the gamma
function its formula holds also for $0<H<\frac 1 2$.   
\end{proof}

\begin{lema}\label{sec:lemainjectivity}
Let $X_t$ be the fOU process and $\theta=(H,\lambda,\sigma)$. Consider the mapping:
\begin{align*}
  \nns{\rho}_{\theta,3}(\alpha)=(\rho_\theta(\alpha \cdot 0),\rho_\theta(\alpha \cdot 1),\rho_\theta(\alpha \cdot 2))'.
\end{align*}
in a closed rectangle $\Upsilon\subset \Theta$ where it holds that $\lambda<\exp(\Psi(3))$ and $\sigma$ is of known sign. Then
Assumption \ref{sec:ap1} is satisfied for $\alpha$
sufficiently small\footnote{$\Psi(x)=\frac{d}{dx}\log \Gamma (x)$: digamma
function. Note that $\exp(\Psi(3))\approx 2.52$.}.
\begin{proof}
  Following the proof of Lemma 5.2 in \cite{Hu} we have that:
  \begin{align*}
    \rho_\theta(t)\approx \sigma^2\lambda^{-2H}H\Gamma(2H)-\frac{\sigma^2t^{2H}}{2}
  \end{align*}
for $t$ sufficiently small. Define:
\begin{align*}
  h_1(\theta)&:=\rho_\theta(0)=\frac{\sigma^2}{2}\lambda^{-2H}\Gamma(2H+1)\\
  h_2(\theta)&:=\rho_\theta(\alpha)=\frac{\sigma^2}{2}\lambda^{-2H}\Gamma(2H+1)-\frac{\sigma^2}{2}\alpha^{2H}\\
h_3(\theta)&:=\rho_\theta(2\alpha)=\frac{\sigma^2}{2}\lambda^{-2H}\Gamma(2H+1)-\frac{\sigma^2}{2}(2\alpha)^{2H}.
\end{align*}
It is not hard to verify that if $\lambda<\exp(\Psi(2H+1))<\exp(\Psi(3))$ then
$\frac{\partial h_1}{\partial H}>0$. Also it is straightforward to compute
that:
\begin{align*}
  \begin{vmatrix}
    \frac{\partial h_1}{\partial H} & \frac{\partial h_1}{\partial \lambda}\\
    \frac{\partial h_2}{\partial H} & \frac{\partial h_2}{\partial \lambda}
  \end{vmatrix}=-H\sigma^4\Gamma(2H+1)\lambda^{-2H-1}\alpha^{2H}\log \alpha
\end{align*}
and the latter is strictly positive iff $\alpha<1$. Finally:
\begin{align*}
  \begin{vmatrix}
        \frac{\partial h_1}{\partial H} & \frac{\partial h_1}{\partial
          \lambda} & \frac{\partial h_1}{\partial
          \sigma}\\
    \frac{\partial h_2}{\partial H} & \frac{\partial h_2}{\partial \lambda} & \frac{\partial h_2}{\partial
          \sigma}\\
    \frac{\partial h_3}{\partial H} & \frac{\partial h_3}{\partial \lambda} & \frac{\partial h_3}{\partial
          \sigma}
  \end{vmatrix}=-H\sigma^5\Gamma(2H+1)\lambda^{-2H-1}2^{2H}\alpha^{4H}\log 2.
\end{align*}
and this expression is positive if $\text{sign}(\sigma)<0$. Using the Theorem 4 in
\cite{Gale1965} we can deduce that the mapping $\nns{\rho}_{\theta,3}(\alpha)$
is injective for $\alpha$ small. If the
$\text{sign}(\sigma)>0$ we can deduce the injectivity from the autocovariance
function $\rho_\theta$ related to the process $-X$. Also note that the mapping satisfies the
remark after 
assumption \ref{sec:ap1} as a direct consequence of the above calculations.
\end{proof}
\end{lema}

\section{Behavior of $(\hat {\nn g}_N(\theta))_0$ when $H\geq 3/4$.}
\label{sec:behavior-g_0-when}

The last result we present and prove in this paper, Proposition
\ref{sec:fOUgbehavior} below, shows that if one does not perform any filtering
of the observations of an fOU process, but one attempts to construct a
Method-of-Moments estimator based dierctly its the quadratic variations, for
the purpose of estimating its variance, for instance, then the estimator
cannot be asymptotically normal when $H>3/4$, or even asymptotically of
second-chaos type in law; this is the first item in the next
proposition. However, the second item in this proposition shows that the
estimator is strongly consistent nonetheless, with an $L^2$-rate of
convergence of order $N^{2H-2}$ (which is slower than $N^{-1/2}$). The
standard Borel-Cantelli argument (such as the one employed to prove Lemma \ref{sec:LemacotaEPsi}) then implies that the estimator is strongly consistent with rate $N^{-\gamma}$ for any $\gamma<2-2H$. The proof of Proposition \ref{sec:fOUgbehavior} requires elements of the Malliavin calculus, which are also given below.

\begin{prop}\label{sec:fOUgbehavior}
  Assume that $X_t$ is a stationary fOU process with parameters $\theta=(H,\lambda,\sigma)$ and $H\geq \frac 3 4$. Then:
  \begin{itemize}
  \item[(i)] If $H=\frac 3 4$:
    \begin{align*}
      \sqrt{\frac{N}{\log N}}(\hat {\nn g}_N(\theta))_0\stackrel{d}{\longrightarrow}N(0,2\alpha^{-1}c_\theta^2)
    \end{align*}
    where $c_\theta=\sigma^2\frac{H(2H-1)}{\lambda}=\frac{3\sigma^2}{8\lambda}$.
  \item[(ii)] If $H>\frac 3 4$, $(\hat {\nn g}_N(\theta))_0$ does not converge to a normal law, or even a second-chaos law. However, 
    \begin{align*}
      E\left[|N^{2-2H}(\hat {\nn g}_N(\theta))_0|^2\right]=O(1).
    \end{align*}
  \end{itemize}
  \begin{proof}
    First note that since $SA_{i,0}=A_i$ for all $i
    \in\{0,\ldots,N\}$ (see equations
    \eqref{eq:9} and \eqref{eqphiell}), the variance of $\hat g_N(\theta):=(\hat {\nn g}_N(\theta))_0$ is:
    \begin{align*}
      E[|\hat g_N(\theta)|^2]&=\frac{2}{N^2}\sum_{i,j=1}^N\langle A_i\otimes A_i,A_j\otimes A_j\rangle_{\Hcal}\\
      &=\frac{2}{N^2}\sum_{i,j=1}^N\langle A_i,A_j\rangle_{\Hcal}^2=\frac{2}{N^2}\sum_{i,j=1}^N[\rho_\theta(|t_i-t_j|)]^2\\
      &=\frac{2\rho_\theta(0)^2}{N}+\frac{2}{N^2}\sum_{\substack{i,j=1 \\ i\neq j}}^N[\rho_{\theta_0}(|t_i-t_j|)]^2\\
      &=\frac 2 N\underbrace{\left[\rho_\theta(0)^2+\rho_\theta(1)^2\right]}_{\beta_\theta}+\frac{2}{N^2}\sum_{\substack{i,j=1 \\ |i-j|\geq 2}}^N[\rho_{\theta_0}(|t_i-t_j|)]^2
    \end{align*} 
    where
    $\Hcal:=L^2\left[\left(-\frac{\pi}{\alpha},\frac{\pi}{\alpha}\right)\right]$
    and its inner product is given by:
\begin{align*}
\langle A_j,A_{j'}\rangle_\Hcal=\rho_\theta(|t_j-t_{j'}|)=E[X_{t_j}X_{t_{j'}}]=E[I_1(A_j(\cdot))I_1(A_{j'}(\cdot))].
\end{align*}
Using the expansion (\ref{expancorr}) we have:
    \begin{align}\label{eq:2}
      &E[|\hat g_N(\theta)|^2]=\frac{2\beta_\theta}{N}+\frac{2}{N^2}\sum_{\substack{i,j=1\\ |i-j|\geq 2}}^N\left[c_\theta^2|t_i-t_j|^{4H-4}+O(|t_i-t_j|^{4H-6})\right]\\
      &=\frac{2\beta_\theta}{N}+\frac 2 N \left[\alpha^{4H-4}c_\theta^2\sum_{k=2}^N\left(\frac{N-k}{N}\right)k^{4H-4}+\alpha^{4H-6}\sum_{k=2}^N\left(\frac{N-k}{N}\right)O(k^{4H-6})\right]\notag
    \end{align}
    where $c_\theta:=\sigma^2\frac{H(2H-1)}{\lambda}$. Note that for $H=\frac 3 4$, as $N\rightarrow \infty$:
    \begin{align*}
      E\left[\frac{N}{\log N}|\hat g_N(\theta)|^2\right]\longrightarrow 2\alpha^{-1}c_\theta^2
    \end{align*}
    We can write (\ref{eq:2}) when $H>\frac 3 4$ as:
    \begin{align*}
      &E[|\hat g_N(\theta)|^2]=\frac{2\beta_\theta}{N}+\frac{2}{N^2}\sum_{\substack{i,j=1\\ |i-j|\geq 2}}^N\left[c_\theta^2|\alpha(i-j)|^{4H-4}+O(|\alpha(i-j)|^{4H-6})\right]\\
      &=\frac{2\beta_\theta}{N}+2\alpha^{4H-4}c_\theta^2N^{4H-4}\sum_{\substack{i,j=1\\ |i-j|\geq 2}}^N\frac{|i-j|^{4H-4}}{N^{4H-4}}\frac{1}{N^2}+\frac 2 N \alpha^{4H-6}\sum_{k=2}^N\left(\frac{N-k}{N}\right)O(k^{4H-6}) 
    \end{align*}
    therefore:
    \begin{align*}
      E[|N^{2-2H}\hat g_N(\theta)|^2]\longrightarrow 2\alpha^{4H-4}c_\theta^2 \iint_{[0,1]^2}|x-y|^{4H-4}dxdy=\frac{2\alpha^{4H-4}c_\theta^2}{(2H-1)(4H-3)} 
    \end{align*}
    Define:
    \begin{align*}
      \tilde D_N(\theta):=
      \begin{cases}
        \sqrt{\frac{N}{\tilde c_1\log N}}& \text{if $H=\frac 3 4$}\\
        \sqrt{\frac{N^{4-4H}}{\tilde c_2}}& \text{if $H>\frac 3 4$}
      \end{cases}
    \end{align*}
    where $\tilde c_1=2\alpha^{-1}c_\theta^2$ and $\tilde c_2=\frac{2\alpha^{4H-4}c_\theta^2}{(2H-1)(4H-3)}$ and denote $F_N:=\tilde D_N(\theta)\hat g_N(\theta)$. For both cases we have that $E[|F_N|^2]\rightarrow 1$ and hence the main assumption of Theorem 4 in \cite{nualartortiz} is satisfied. Now we only need to prove that:
    \begin{align*}
      \|DF_N\|^2_{\Hcal}\stackrel{L^2(\Omega)}{\longrightarrow}2
    \end{align*}
    in order to prove the asymptotic convergence of $F_N$ towards a normal law. We can compute the Malliavin derivative of $\hat g_N(\theta)$ as follows:
    \begin{align*}
      D_r\hat g_N(\theta)=\frac 1 N\sum_{i=1}^N2X_{t_i}D_rX_{t_i}=\frac{1}{N}\sum_{i=1}^NI_1(A_i(\cdot))A_i(r)
    \end{align*}
    and hence it turns out that:
    \begin{align*}
      E[\|D\hat g_N(\theta)\|^2_{\Hcal}]&=\frac{4}{N^2}\sum_{i,j=1}^NE[I_1(A_i(\cdot))I_1(A_j(\cdot))]\langle A_i,A_j\rangle_{\Hcal}\\
      &=\frac{4}{N^2}\sum_{i,j=1}^N\langle A_i,A_j\rangle_{\Hcal}^2=2E[\|\hat g_N(\theta)\|_{\Hcal}^2]
    \end{align*}
    and this implies that:
    \begin{align*}
      E[\|DF_N\|_{\Hcal}^2]=2E[\|F_N\|_{\Hcal}^2]\longrightarrow 2.
    \end{align*}
    Therefore, we only need to prove that:
    \begin{align*}
      \|DF_N\|_{\Hcal}^2-E[\|DF_N\|_{\Hcal}^2]\stackrel{L^2(\Omega)}{\longrightarrow}0
    \end{align*}
    First note that by using the multiplication rule of Weiner integrals:
    \begin{align*}
      \|DF_N\|_{\Hcal}^2&=\frac{4}{N^2}\tilde D_N(\theta)^2\sum_{i,j=1}^NI_1(A_i(\cdot)) I_1(A_j(\cdot))\langle A_i,A_j\rangle_{\Hcal}\\
      &=\frac{4}{N^2}\tilde D_N(\theta)^2\sum_{i,j=1}^N\left[I_2(A_i(\cdot)\otimes A_j(\cdot))+\langle A_i,A_j\rangle_{\Hcal}\right]\langle A_i,A_j\rangle_{\Hcal}\\
      &=\frac{4}{N^2}\tilde D_N(\theta)^2\sum_{i,j=1}^N I_2(A_i(\cdot)\otimes A_j(\cdot))+E[\|DF_N\|_\Hcal^2]
    \end{align*}
    and hence using the Lemma \ref{sec:some-auxil-results} (i) and (ii) we have that:
    \begin{multline*}
      E\left[\|DF_N\|_{\Hcal}^2-E[\|DF_N\|_{\Hcal}^2]\right]\\=\frac{16}{N^4}\tilde D_N(\theta)^4\sum_{i,j,i',j'=1}^N E[I_2(A_i(\cdot)\otimes A_j(\cdot))I_2(A_{i'}(\cdot)\otimes A_{j'}(\cdot))]\langle A_i,A_j\rangle_{\Hcal} \langle A_{i'},A_{j'}\rangle_{\Hcal}\\
      =\frac{32}{N^4}\tilde D_N(\theta)^4\sum_{i,j,i',j'=1}^N \langle A_i\tilde{\otimes} A_j,A_{i'}\tilde{\otimes} A_{j'}\rangle_{\Hcal^2}\langle A_i,A_j\rangle_{\Hcal} \langle A_{i'},A_{j'}\rangle_{\Hcal}
    \end{multline*}
    and therefore, using Lemma \ref{sec:some-auxil-results-1}, we conclude that:
    \begin{align*}
      E\left[\|DF_N\|_{\Hcal}^2-E[\|DF_N\|_{\Hcal}^2]\right]\longrightarrow 0
    \end{align*}
    only if $H=\frac 3 4$. Then for $H=\frac 3 4$ we can apply the Theorem 4 in \cite{nualartortiz} and deduce that:
    \begin{align*}
      \sqrt{\frac{N}{\tilde c_1\log N}}\hat g_N(\theta)\stackrel{d}{\rightarrow} N(0,1).
    \end{align*}

    For $H>\frac 3 4$, Lemma \ref{sec:some-auxil-results-1} allows us to conclude that the limit of 
    \begin{align*}
      \sqrt{\frac{N^{4-4H}}{\tilde c_2}}\hat g_N(\theta)
    \end{align*}
    does not belong to the first Weiner chaos, and hence the GMM estimator is not normally distributed.
Define the following function:
\begin{align}\label{eq:8}
  \mathcal I_N(r,s):=\frac{N^{1-2H}}{\sqrt{\tilde c_2}}\sum_{j=1}^N (A_j\otimes A_j)(r,s)
\end{align}
for any $(r,s) \in \mathbb R^2$. By Lemma \ref{sec:some-auxil-results-3}, this function is not a Cauchy sequence in $\Hcal^2$. Hence we cannot find a second-Wiener integral $I_2(g(\cdot,\cdot))$ such that:
\begin{align*}
  E[|F_N-I_2(g(\cdot,\cdot))|^2]=2\|\mathcal I_N(\cdot,\cdot)-g(\cdot,\cdot)\|_{\Hcal^2}^2\longrightarrow 0
\end{align*}
when $H>\frac 3 4$ (see the proof of Theorem 2 in \cite{TV2}). Then the limit in $L^2(\Omega)$ of $\hat g_N(\theta)$ is not even a second-chaos random variable.
  \end{proof}
\end{prop}

\begin{lema}\label{sec:some-auxil-results}
  If $A_k(u):=A_k(u|\theta)$ is defined as (\ref{eq:9}) for $u>0$, then:
  \begin{enumerate}
  \item[(i)] The tensor product $(A_k\otimes A_j)(u,v)$ is symmetric if and
    only if $k=j$.
  \item[(ii)] $$\langle A_{i_1}\otimes \cdots \otimes A_{i_n},A_{j_1}\otimes \cdots \otimes A_{j_n}\rangle_{\Hcal^n}=\prod_{k=1}^n\langle A_{i_k},A_{j_k}\rangle_{\Hcal}$$ for any $H\in (0,1)$.
  \end{enumerate}
\end{lema}
\begin{proof}
  \begin{enumerate}
  \item[(i)] It is trivial to deduce that for $u,v>0$:
  \begin{align*}
    (A_k\otimes A_j)(u,v)&=\exp(it_ku)\bar f_{\theta,\alpha}(u)^{1/2}\exp(it_jv)\bar f_{\theta,\alpha}(v)^{1/2}\\
    &=\exp(it_kv)\bar f_{\theta,\alpha}(v)^{1/2}\exp(it_ju)\bar f_{\theta,\alpha}(u)^{1/2}=(A_i\otimes A_j)(v,u)
  \end{align*}
if and only if $\exp(it_k(u-v))=\exp(it_j(u-v))$, which is equivalent to have
that $k=j$.
\item[(ii)] 
  \begin{align*}
    &\langle A_{i_1}\otimes \cdots \otimes A_{i_n},A_{j_1}\otimes \cdots \otimes A_{j_n}\rangle_{\Hcal^n}=\\
&=\int_{(-\pi/\alpha,\pi/\alpha)^n} A_{i_1}(s_1)\cdots A_{i_n}(s_n)\overline{A_{j_1}(s_1)\cdots A_{j_n}(s_n)}ds_1\cdots ds_n\\
&=\left[\int_{-\pi/\alpha}^{\pi/\alpha} A_{i_1}(s_1)\bar A_{j_1}(s_1)ds_1\right]\cdots \left[\int_{-\pi/\alpha}^{\pi/\alpha} A_{i_n}(s_n)\bar A_{j_n}(s_n)ds_n\right]\\
&=\prod_{k=1}^n\langle A_{i_k},A_{j_k}\rangle_{\Hcal}
  \end{align*}
  \end{enumerate}
\end{proof}
\begin{lema}\label{sec:some-auxil-results-1}
  Suppose that $A_k$ is defined as (\ref{eq:9}). Define for $H\geq \frac 3 4$: $$\mathcal P_N:=\frac{\tilde D_N(\theta)^4}{N^4}\sum_{i,j,i',j'=1}^N \langle A_i\tilde{\otimes} A_j,A_{i'}\tilde{\otimes} A_{j'}\rangle_{\Hcal^2}\langle A_i,A_j\rangle_{\Hcal} \langle A_{i'},A_{j'}\rangle_{\Hcal}$$ hence:
  \begin{align*}
    \mathcal P_N\longrightarrow 
    \begin{cases}
      0 & \text{if $H=\frac 3 4$}\\
      k & \text{if $H>\frac 3 4$}.
    \end{cases}
  \end{align*}
  for some $k\neq 0$.
  \begin{proof}
    First note that the inner product of the symmetric tensor products can be written as:
    \begin{align*}
      \langle A_i\tilde{\otimes} A_j,A_{i'}\tilde{\otimes} A_{j'}\rangle_{\Hcal^2}&=\frac 1 4\bigl[\langle A_i\otimes A_j,A_{i'}\otimes A_{j'}\rangle_\Hcal+\langle A_i\otimes A_j,A_{j'}\otimes A_{i'}\rangle_\Hcal+\\
      &\quad \langle A_j\otimes A_i,A_{i'}\otimes A_{j'}\rangle_\Hcal+\langle A_j\otimes A_i,A_{j'}\otimes A_{i'}\rangle_\Hcal\bigl]\\
      &=\frac 1 2 \langle A_i,A_{i'}\rangle_\Hcal\langle A_j,A_{j'}\rangle_\Hcal+\frac 1 2 \langle A_i,A_{j'}\rangle_\Hcal\langle A_j,A_{j'}\rangle_\Hcal
    \end{align*}
    then, by symmetry of the indices, we can write:
    \begin{align*}
      \mathcal P_N:=\frac{\tilde D_N(\theta)^4}{N^4}\sum_{i,j,i',j'=1}^N \langle A_i,A_{i'}\rangle_\Hcal \langle A_j,A_{j'}\rangle_{\Hcal}\langle A_i,A_j\rangle_{\Hcal} \langle A_{i'},A_{j'}\rangle_{\Hcal}
    \end{align*}
    and we can study the limiting behavior of $\mathcal P_N$ through the folllowing cases:
    \begin{description}
    \item[Case I:] $i=i'=j=j'$

      In this case we analyzed the behavior of the diagonal, which is convergent for any $H\geq \frac 3 4$ as follows:
      \begin{align*}
        \mathcal P_N=\frac{\tilde D_N(\theta)^4}{N^4}N\rho_\theta(0)^4&=
        \begin{cases}
          \frac{\rho_\theta(0)^4}{2\tilde c_1^2N\log N} & \text{if $H=\frac 3 4$}\\
          \frac{\rho_\theta(0)^4}{\tilde c_2^2N^{8H-5}} & \text{if $H>\frac 3 4$} 
        \end{cases}\\
        &\longrightarrow 0.
      \end{align*}
    \item[Case II:] $i=i'$, $j=j'$ and $i\neq j$

      In this particular case we observe:
      \begin{multline*}
        \mathcal P_N=\frac{\tilde D_N(\theta)^4}{N^4}\rho_\theta(0)^2\sum_{\substack{i,j=1\\ i\neq j}}^N\langle A_i,A_j\rangle_{\Hcal}^2=\frac{\tilde D_N(\theta)^4\rho_\theta(0)^2}{N^4}\biggl[2(N-1)\rho_\theta(1)^2\\ +\sum_{|i-j|\geq 2}c_\theta^2 \alpha^{4H-4}|i-j|^{4H-4}+c_\theta\alpha^{4H-6}N\sum_{k=2}^N\left(\frac{N-k}{N}\right)O(k^{4H-6})\biggl]
      \end{multline*}
      Note that if $H=\frac 3 4$:
      \begin{align*}
        \mathcal P_N&=\frac{\tilde D_N(\theta)^4\rho_\theta(0)^2}{N^4}\biggl[2(N-1)\rho_\theta(1)^2 +c_\theta^2 \alpha^{4H-4}\sum_{k=2}^N\left(\frac{N-k}{N}\right)k^{4H-4}+\\
&\qquad c_\theta\alpha^{4H-6}N\sum_{k=2}^N\left(\frac{N-k}{N}\right)O(k^{4H-6})\biggl]\\
&\asymp \frac{1}{N\log N}+\frac 1 N\longrightarrow 0.
      \end{align*}
      and if $H>\frac 3 4$:
      \begin{align*}
        \mathcal P_N&=\frac{\tilde D_N(\theta)^4\rho_\theta(0)^2}{N^4}\biggl[2(N-1)\rho_\theta(1)^2 +c_\theta^2 \alpha^{4H-4}\sum_{|i-j|\geq 2}\frac{|i-j|^{4H-4}}{N^{4H-4}}\frac{1}{N^2}+\\
    &\qquad c_\theta\alpha^{4H-6}N\sum_{k=2}^N\left(\frac{N-k}{N}\right)O(k^{4H-6})\biggl]\\
    &\asymp N^{5-8H}+N^{2-4H}\iint_{[0,1]^2}|x-y|^{4H-4}dxdy\longrightarrow 0
      \end{align*}
      since $\iint_{[0,1]^2}|x-y|^{4H-4}dxdy<\infty$ if $H>\frac 3 4$.
    \item[Case III:] $i=i'$, $j\neq j'$ and $i\neq j$

      Without loss of generality, we can assume that $|i-j|\geq 2$ and $|j-j'|$. If $|i-j|=1$ or $|j-j'|=1$ then it can be proved that the rates of convergence are the same as Case I and II. We can write $\mathcal P_N$ as:
      \begin{align*}
        \mathcal P_N&=\frac{\tilde D_N(\theta)^4}{N^4}\rho_\theta(0)\biggl[\overbrace{c_\theta^3\alpha^{6H-6}\sum_{i,j,j'=1}^N|i-j|^{2H-2}|i-j'|^{2H-2}|j-j'|^{2H-2}}^{\mathcal A_N}\\
&+\underbrace{3c_\theta^2\alpha^{6H-8}\sum_{i,j,j'=1}^N|i-j|^{2H-4}|i-j'|^{2H-2}|j-j'|^{2H-2}O(1)}_{\mathcal B_N}\biggl]
      \end{align*}
      the first factor behaves as:
      \begin{align*}
        \mathcal A_N&=c_\theta^3\alpha^{6H-3}N^{6H-3}\sum_{i,j,j'=1}^N\left|\frac{i-j}{N}\right|^{2H-2}\left|\frac{i-j'}{N}\right|^{2H-2}\left|\frac{j-j'}{N}\right|^{2H-2}\frac{1}{N^3}\\
        &\asymp N^{6H-3}\iiint_{[0,1]^3}|x-y|^{2H-2}|x-y'|^{2H-2}|y-y'|^{2H-2}dxdydy'
      \end{align*}
      and the second:
      \begin{align*}
        \mathcal B_N&=c_\theta^3\alpha^{6H-3}N^{6H-5}\sum_{i,j,j'=1}^N\left|\frac{i-j}{N}\right|^{2H-2}\left|\frac{i-j'}{N}\right|^{2H-4}\left|\frac{j-j'}{N}\right|^{2H-2}\frac{1}{N^3}\\
        &\leq c_\theta^3\alpha^{6H-3}2^{-2}N^{6H-3}\sum_{i,j,j'=1}^N\left|\frac{i-j}{N}\right|^{2H-2}\left|\frac{i-j'}{N}\right|^{2H-2}\left|\frac{j-j'}{N}\right|^{2H-2}\frac{1}{N^3}\\
        &\asymp N^{6H-3}\iiint_{[0,1]^3}|x-y|^{2H-2}|x-y'|^{2H-2}|y-y'|^{2H-2}dxdydy'
      \end{align*}
      and the previous integral is finite for any $H>\frac 1 2$. Hence, $\mathcal P_N$ is bounded asymptotically, for $H=\frac 3 4$, by:
      \begin{align*}
        \frac{N^2\cdot N^{6H-3}}{N^4\log N}=\frac{1}{N^{5-6H}\log N}\longrightarrow 0
      \end{align*}
      and if $H>\frac 3 4$, it is bounded by:
      \begin{align*}
        \frac{N^{8-8H}\cdot N^{6H-3}}{N^4}=N^{1-2H}\longrightarrow 0
      \end{align*}
    \item[Case IV:] $i\neq i'$, $j\neq j'$ and $i\neq j$

      As we did in the previous case, the rates of convergence when at least two indices are far apart by less than two units are the same as the ones calculated in Cases I-II-III. So let us assume that $|i-i'|\geq 2$, $|j-j'|\geq 2$ and $|i-j|\geq 2$.
      \begin{align*}
        \mathcal P_N&=\frac{\tilde D_N(\theta)^4}{N^4}\biggl[\overbrace{c_\theta^4\alpha^{8H-8}\sum_{i,j,i',j'=1}^N|i-j|^{2H-2}|i'-j'|^{2H-2}|i-j'|^{2H-2}|i'-j|^{2H-2}}^{\mathcal C_N}\\
&+\underbrace{4c_\theta^3\alpha^{8H-10}\sum_{i,j,i',j'=1}^N|i-j|^{2H-2}|i'-j'|^{2H-2}|i-j'|^{2H-2}|i'-j|^{2H-4}}_{\mathcal D_N}\biggl]
      \end{align*}
      and note that the first summand behaves asymptotically as:
      \begin{align*}
        \mathcal C_N&=c_\theta^4\alpha^{8H-8}N^{8H-4}\sum_{i,j,i',j'=1}^N\left|\frac{i-j}{N}\right|^{2H-2}\left|\frac{i'-j'}{N}\right|^{2H-2}\left|\frac{i-j'}{N}\right|^{2H-2}\left|\frac{i'-j}{N}\right|^{2H-2}\frac{1}{N^4}\\
&\asymp N^{8H-4}\int_{[0,1]^4}(x-y)^{2H-2}(x'-y)^{2H-2}(y-y')^{2H-2}(x-x')^{2H-2}dxdydx'dy'
      \end{align*}
      and the second summand can be asymptotically bounded by:
      \begin{align*}
        c_\theta^3\alpha^{8H-10}N^{8H-4}\int_{[0,1]^4}(x-y)^{2H-2}(x'-y)^{2H-2}(y-y')^{2H-2}(x-x')^{2H-2}dxdydx'dy'
      \end{align*}
      and this quadruple integral is bounded for any $H>\frac 1 2$. Hence if $H=\frac 3 4$, $\mathcal P_N$ is asymptotically bounded by:
      \begin{align*}
        \frac{N^2}{N^4\log N}N^{8H-4}=\frac{1}{\log N}\longrightarrow 0.
      \end{align*}
      However, if $H>\frac 3 4$, the first summand of $\mathcal P_N$ is asymptotically equal to:
      \begin{align*}
        \frac{N^{8-8H}N^{8H-4}}{N^4}=1
      \end{align*}
      then the limit of $\mathcal P_N$ is not 0. Moreover, since $\mathcal D_N$ is asymptotically bounded by $N^{8H-4}$, there exists $k\neq 0$ such that:
      \begin{align*}
        \mathcal P_N\longrightarrow k.
      \end{align*}
    \end{description}
  \end{proof}

\end{lema}
\begin{lema}\label{sec:some-auxil-results-3}
  If $\mathcal I_N$ is defined as (\ref{eq:8}) and $H>\frac 3 4$, then it is not a Cauchy sequence in $\Hcal^2$.
  \begin{proof}
    We will prove the statement by contradiction. Assume that $\mathcal I_N$ is Cauchy. First of all, note that $\mathcal I_N$ is bounded in $\Hcal^2=L^2((-\pi,\pi)^2)$:
    \begin{align*}
      \|\mathcal I_N(\cdot,\cdot)\|_{\Hcal^2}^2&=\frac{N^{2-4H}}{\tilde c_2}\sum_{i,j=1}^N\langle A_i \otimes A_i, A_j \otimes A_j\rangle_{\Hcal^2}=\frac{N^{2-4H}}{\tilde c_2}\sum_{i,j=1}^N\langle A_i, A_j\rangle_{\Hcal}^2\\
      &=\frac{N^{4-4H}}{\tilde c_2}\frac{1}{N^2}\left[N\rho_\theta(0)+2(N-1)\rho_\theta(1)+\sum_{|i-j|\geq 2}[\rho_\theta(t_i-t_j)]^2\right]\\
      &\leq M_1N^{3-4H}+\frac{N^{2-4H}}{\tilde c_2}\sum_{|i-j|\geq 2}[c_\theta[\alpha(i-j)^{2H-2}]+O(\alpha(i-j)^{2H-4})]^2\\
      &=M_1N^{3-4H}+\frac{1}{\tilde c_2}c_\theta^2\alpha^{4H-4}\sum_{|i-j|\geq 2}\left|\frac{i-j}{N}\right|^{4H-4}\frac{1}{N^2}\\
& \qquad+\alpha^{4H-6}N^{3-4H}\sum_{k=2}^N\left(\frac{N-k}{N}\right)k^{4H-6}\\
&<M_1+\frac 1 2+\alpha^{4H-6}M_2<\infty
    \end{align*}
    where $M_1:=\frac{\rho_\theta(0)+2\rho_\theta(1)}{\tilde c_2}$ and $M_2:=\sum_{k=1}^\infty k^{4H-6}$. This implies that $\mathcal I_N$ is dominated in $L^2((-\pi/\alpha,\pi/\alpha)^2)$. 

Also note that the pointwise limit of $\mathcal I_N$ is (if $H>\frac 1 2$): 
    \begin{align*}
      |\mathcal I_N|&=\left|\frac{N^{1-2H}}{\sqrt{\tilde c_2}}\sum_{j=1}^N \exp(irt_j)\exp(ist_j)\bar f_{\theta,\alpha}(r)\bar f_{\theta,\alpha}(r)\right|\\
      &=\frac{N^{1-2H}}{\sqrt{\tilde c_2}}\bar f_{\theta,\alpha}(r)\bar f_{\theta,\alpha}(r)\left|\sum_{j=1}^N \exp(i\alpha rj)\exp(i\alpha sj)\right|
      \asymp N^{1-2H}\longrightarrow 0
    \end{align*}
    since:
    \begin{align*}
      \left|\sum_{j=1}^N \exp(i\alpha rj)\exp(i\alpha
        sj)\right|&<\left|\sum_{j=0}^\infty\exp(i\alpha rj)\exp(i\alpha
        sj)-1\right|\\
      &\leq \left|\frac{\exp[i\alpha(r+s)]}{1-\exp[i\alpha(r+s)]}\right|<\infty.
    \end{align*}
Then we can apply the dominated convergence theorem to conclude that:
    \begin{align*}
      E\left[\left|\frac{N^{2-2H}}{\sqrt{\tilde c_2}}\hat g_N(\theta)\right|^2\right]=E\left[|I_1(\mathcal I_N(\cdot,\cdot))|^2\right]=2\|\mathcal I_N(\cdot ,\cdot)\|_{\Hcal^2}^2\longrightarrow 0
    \end{align*}
    but this contradicts the fact that:
    \begin{align*}
      E\left[\left|\frac{N^{2-2H}}{\sqrt{\tilde c_2}}\hat g_N(\theta)\right|^2\right]\longrightarrow 1
    \end{align*}
    which was already evidenced in the proof of Theorem \ref{sec:fract-ornst-uhlenb}. Hence, by contradiction, $\mathcal I_N$ is not Cauchy.  
  \end{proof}
\end{lema}

\end{document}